\documentclass[11pt]{article}
\usepackage[utf8]{inputenc}
\usepackage[english]{babel}
\usepackage{amsmath}
\usepackage{amsthm}
\usepackage{amsfonts}
\usepackage{amssymb}
\usepackage{appendix}
\usepackage{graphicx}
\usepackage{color}
\usepackage{rotating}
\usepackage{authblk}
\usepackage{url}
\usepackage{tcolorbox}
\usepackage{fullpage}
\usepackage[linesnumbered,ruled]{algorithm2e}
\usepackage{tabularx}
\usepackage[caption=false]{subfig}
\usepackage{multirow, float, makecell}
\usepackage{mathrsfs}
\usepackage[labelfont=bf,labelsep=period,justification=raggedright]{caption}
\graphicspath{{images/}}

\newcommand{\R}{\mathbb R}
\newcommand{\N}{\mathbb N}
\newcommand{\nf}{\text{nf}}

\newtheorem{theorem}{Theorem}
\newtheorem{proposition}{Proposition}
\newtheorem{corollary}{Corallary}
\newtheorem{definition}{Definition}
\newtheorem{observation}{Observation}
\newtheorem{lemma}{Lemma}

\usepackage{natbib}
\setcitestyle{authoryear,open={((},close={))}}
\bibpunct[, ]{(}{)}{,}{a}{}{,}%

\allowdisplaybreaks

\begin{document}
%%%%%%%%%%%%%%%%

\title{A Scalable Lower Bound for the Worst-Case Relay Attack Problem on the Transmission Grid}
\author{Emma~S.~Johnson$^{1,2}$, Santanu~S.~Dey$^{2}$\thanks{This work was supported by Sandia National Laboratories' Laboratory Directed Research and Development (LDRD) program.
		Sandia National Laboratories is a multimission laboratory
		managed and operated by National Technology and Engineering
		Solutions of Sandia, LLC, a wholly owned subsidiary
		of Honeywell International, Inc., for the U.S. Department
		of Energy’s National Nuclear Security Administration under
		contract DE-NA0003525. SAND NO. 2021-10211 O.
		The views expressed in the article do not necessarily represent the views of the U.S. Department of Energy or the United States Government.} \\
	\small $^{1}$School of Industrial and Systems Engineering, Georgia Institute of Technology, Atlanta, GA, USA. \\
	\small ejohnson335@gatech.edu, santanu.dey@isye.gatech.edu \\
	\small $^{2}$Sandia National Laboratories, Albuquerque, NM, USA \\
}

\maketitle

\begin{abstract}
We consider a bilevel attacker-defender problem to find the worst-case attack on the relays that control transmission grid components.
The attacker infiltrates some number of relays and renders all of the components connected to them inoperable, with the goal of maximizing load shed.
The defender responds by minimizing the resulting load shed, re-dispatching using a DC optimal power flow (DCOPF) problem on the remaining network.
Though worst-case interdiction problems on the transmission grid have been studied for years, there remains a need for exact and scalable methods.
Methods based on using duality on the inner problem rely on the bounds of the dual variables of the defender problem in order to reformulate the bilevel problem as a mixed integer linear problem (MILP).
Valid dual bounds tend to be large, resulting in weak linear programming relaxations and hence making the problem more difficult to solve at scale. Often smaller heuristic bounds are used, resulting in a lower bound.
In this work we also consider a lower bound, where instead of bounding the dual variables, we drop the constraints corresponding to Ohm's law, relaxing DCOPF to capacitated network flow.
We present theoretical results showing that, for uncongested networks, approximating DCOPF with network flow yields the same set of injections, and thus the same load shed, which suggests that this restriction likely gives a high-quality lower bound in the uncongested case.
Furthermore, we show that in the network flow relaxation of the defender problem, the duals are bounded by 1, so we can solve our restriction exactly.
Last, because the big-M values in the linearization are equal to 1 and network flow has a well-known structure, we see empirically that this formulation scales well computationally with increased network size.
Through empirical experiments on 16 networks with up to 6468 buses, we find that this bound is almost always as tight as we can get from guessing the dual bounds, even for congested networks where the theoretical results do not hold. In addition, calculating the bound is approximately 150 times faster than achieving the same bound with the reformulation guessing the dual bounds.
\end{abstract}


\section{Introduction}\label{sec:intro}
As the power grid becomes increasingly decentralized and networked, so does the potential for damaging cyber attacks.
As is pointed out in \cite{IdahoReport}, the United States electric grid was not originally designed to be networked, and control systems continue to become more complex.
Both frequency and severity of cyber attacks on the grid have increased in the United States, and as smart grid capabilities continue to expand, distributed control of the grid could introduce additional vulnerabilities.
Unlike physical attacks, which are more likely to affect a localized region of the grid, cyber attacks have the potential to infiltrate control centers, meaning an adversary could gain control over operations spanning large portions of the grid.
For example, this was seen in the December 2015 cyber attack on Ukraine's power system, which affected about 225,000 customers.
The attackers took control of an interface which let them open breakers, directly cutting power to customers. 
Simultaneously, they rendered the communication system useless with a denial of service attack (\cite{SunHL2018}).
It is therefore increasingly critical to protect the grid against large and geographically disparate attacks. 

In this paper, we study a bilevel optimization problem which seeks to determine a worst-case cyber attack on the transmission grid.
Such a model could be used to assess the vulnerability of the grid, and to decide portions of the grid which should be further protected, or perhaps air-gapped from the rest of the network.
Additionally, this bilevel problem could be a subproblem in long-term planning problems such as physical and cyber network design problems seeking to minimize the risk of cyber attacks.
The model we solve was originally presented in \cite{CastilloACS2019}.
We will refer to it as the {\it worst-case relay attack problem}.
The attacker, seeking to maximize load shed, can choose a number of relays to infiltrate, constrained by a budget parameter.
All of the physical grid components these relays control are rendered inoperable, and the defender redispatches by solving a DC optimal power flow (DCOPF) model in order to minimize load shed.

Several prior works have considered a similar model, though in these models, the attacker targets grid components directly rather than relays.
Note that the worst-case relay attack problem is no more difficult than interdiction problems which consider direct attacks on grid components.
The problem of a worst-case grid attack was introduced in \cite{SalmeronWB2004}, and solved with a heuristic version of Benders decomposition.
The subsequent work on methodology includes both exact and heuristic methods, but the exact methods rely on having strong bounds on the duals of the DCOPF linear program in order to scale in network size and in the size of the attacker's budget.
	Thus, it is common to use heuristic bounds on the dual variables in these methods, and solve the resulting restriction of the original problem.

In \cite{alvarez} and \cite{SalmeronWB2004a}, the authors propose replacing the defender problem with its dual, linearizing the resulting bilinear terms, and solving the problem as a mixed-integer linear program (MILP).
This technique only works for interdiction problems (that is, it only works on min-max problems), and for the method to be exact, the linearization requires valid bounds on the dual variables of the defender problem. 
The authors of \cite{MottoAG2005} adopt the same bilevel formulation, and reformulate the problem by leaving both the primal and dual variables of the inner problem in the model, generalizing the methodology beyond min-max problems and achieving better scalability in terms of attack budget.
\cite{Arroyo2010} considers a worst-case attack problem where only lines can be attacked, comparing the duality-based single-level reformulation with a mathematically equivalent reformulation using the Karush-Kuhn-Tucker (KKT) optimality conditions of the defender problem.
Again, both of these methods are exact with valid bounds on the duals (or lagrange multipliers). 
The author finds that, on a 24-bus system with attack budgets of up to 16 lines, the duality-based method empirically outperforms the KKT-based method by a difference of several orders of magnitude.

In \cite{SalmeronB2009}, the authors present a generalized Benders Decomposition algorithm based on the assumption that the total load shed cannot increase by more than the capacity of any one grid component when that component is attacked.
This algorithm is capable of solving the problem on networks with more than 5000 buses, but the scaling is not shown to accommodate increases in the size of the attack budget. 
For physical attacks, limited attack budgets are likely realistic, but for cyber attacks, attackers are not limited by physical resources and could therefore be able to attack large portions of the grid that might not be geographically correlated.
In addition, the method is only exact if the assumption holds, which is not necessarily true in congested networks.

The authors of \cite{SundarCNB2018} consider a probabilistic version of the problem which they solve with an algorithm similar to the benders approach in \cite{SalmeronB2009}.
	They compare several formulations for power flow in the defender problem, including the network flow restriction we analyze in this paper.
	The computational study explicitly shows the boundaries of tractability in terms both the network size and the attacker budget.
	Their approach scales to networks with up to 2,383 buses, and attack budgets of up to 5 components.
	The authors of \cite{SundarMBP2019} consider yet another variation of the problem in which the attacks are assumed to be spatially or topologically correlated. 
	With a similar benders approach, they are able to solve on networks up to 240 buses with attack budgets up to 6 lines.
	In \cite{SundarMBP2021}, the authors revisit this model and develop a cut generation algorithm based on a penalty-based reformulation.
	In this methodology, the only bounds on the DCOPF needed are bounds on the dual variables corresponding to the thermal limit constraints.
	Some of these are fixed to 0 for lines that can never be at full capacity.
	The authors compare an exact version of their method, where the duals that are not fixed to 0 are bounded by the total load in the system, to a heuristic method where they bound these duals by 1.
	The number of iterations required for the heuristic version to converge tends to be at least an order of magnitude less than the exact method.

\begin{table}
	\centering
	\footnotesize
	\caption{Summary of scalability of previous literature on worst-case attack problem in terms of number of buses in the network as well as cardinality in the attack budget.\label{table:sad-scalability}}
	\begin{tabular}{l | r r}
			\hline 
		Citation & Maximum Network Size & Maximum Attack Budget \\
		\hline
		\cite{SalmeronWB2004a} &  48 buses & 24 components \\
		\cite{SalmeronWB2004} & 48 buses & 40 components \\% This is the IEEE one
		\cite{alvarez} & 24 buses & 6 components\\
		\cite{MottoAG2005} & 48 buses & 40 components\\
		\cite{SalmeronB2009} & $> 5000$ buses & 18 components \\
		\cite{Arroyo2010} & 24 buses & 16 lines \\
		\hline
	\end{tabular}
	{}
\end{table}

In summary, most of the existing methodology requires valid bounds on the dual variables of the DCOPF linear program to be exact. 
	Since these are large, the scalability of exact methods is limited in terms of the size of the network and the size of the attack budget, as is summarized in Table \ref{table:sad-scalability}.
The assumptions in these methods are symptoms of a broader problem in bilevel optimization:
All methods of dualizing the inner problem in order to combine it with the outer problem require relatively tight upper bounds on the dual variables of the inner problem (\cite{SmithS2020}).
Though it is common to use heuristics to calculate big-M values with which to linearize the KKT conditions of the inner problem, \cite{PinedaM2019} show that these heuristics can fail, even for bilevel problems with linear programming leader and follower problems.
Furthermore, \cite{KleinertLPS2019} show that verifying the correctness of big-M values in bilevel optimization is as hard as solving the original problem.
They suggest that, if we choose to solve bilevel problems by reformulating the follower's problem using duality or its KKT conditions, then we will have to resort to problem-specific information in order to generate valid big-M values.
Last, while methods such as covering decomposition from \cite{IsraeliW2002} are both exact and applicable to this problem (and do not require a big-M), note that the cuts to block previously-generated attacks are included in the benders algorithm from \cite{SalmeronB2009}, implying that without enhancement, this is not a scalable approach.

Despite the fact that the scalability of the existing approaches for the worst-case attack model is limited, there has been continued interest in the literature in solving extensions of this model and more complicated models which include this model.
\cite{BienstockV2010} develop a problem-specific algorithm for a variation of the problem where the attacker minimizes the number of lines necessary to attack in order to achieve a prespecified amount of load shed.
In addition, the authors provide a novel model in which the attacker antagonistically modifies the resistances of the power lines.
Further extensions include the addition of transmission line switching as an option for the defender in \cite{DelgadilloAA2010} and \cite{ZhaoZ2013}, inclusion of both short- and medium-term impacts of attacks in \cite{WangB2014}, modeling attacks which unfold over time in \cite{SayyadipourYL2016}, modeling coordinated cyber and physical attacks in \cite{LiSAA2016}, and, as previously mentioned, adding the assumption of spacially correlated physical attacks in \cite{SundarMBP2019} and \cite{SundarMBP2021}.
In addition, there has been interest in trilevel planning problems such as defensive hardening of the network in \cite{YuanZZ2014}, \cite{AlguacilDA2014}, and \cite{WuC2017}.%TODO: add Bryan's paper if it's on arxiv

In this work, we revisit the network flow restriction from \cite{SundarCNB2018}.
That is, instead of solving DCOPF in the defender problem, we drop the Ohm's law constraints, simplifying the inner problem to capacitated network flow.
This is a restriction of the original problem, as it expands the defender's feasible region, thus restricting the attacker's options.
Applying it to get a lower bound for the worst-case relay attack problem, we show it can be used on networks with more than 6000 buses with attack budgets ranging from small numbers of relays up to 30\% of the network, enough to shed all of the load.
While such an approach only gives a lower bound, we formally show that, when line capacities are large enough, the optimal objective value of the network flow restriction is the same as that of the original worst-case relay attack problem. 
This is because network flow is a good approximation of DCOPF when both formulations are projected into the space of injections.
That is, the line flows in the optimal solution of the network flow restriction may be dramatically different from those of DCOPF, but the load shed and generator dispatch will be the same.
Since the formulation measures the severity of the attack in terms of load shed, the accuracy of the line flows will not effect the attack solution unless the network is congested.
While we do expect the attacker to take advantage of his ability to create congestion, we find empirically that, even on congested instances, the bound we get from the network flow restriction is almost always as tight as we can find when we solve the original problem reformulated with improvised dual bounds.

As is observed in \cite{RoaldM2019}, in DCOPF, very few line limit constraints are ever tight, even accounting for variation in both demand and generation costs.
In other words, in practice, transmission networks are rarely congested.
Thus, it is not unexpected that the network flow restriction bound appears to be high quality.
In addition, we find that we can obtain this bound within 20 minutes, even on large-scale networks with difficult-to-solve attack budgets.
Though this is likely because capaciated network flow is a familiar and highly-optimized problem for commercial solvers, it is worth mentioning that network flow interdiction is itself a well-studied problem with some promising theoretical results which might eventually be applied to solve the network flow restriction.
For example, \cite{ChestnutZ2017} give an approximation algorithm for an interdiction problem where the attacker eliminates edges in order to minimize the maximum $s$-$t$ flow.
With slight modifications (i.e., modeling generators and loads with mock lines to a super source and super sink respectively), the network flow restriction can be modeled as a maximum $s$-$t$ flow, so \cite{ChestnutZ2017} and other combinatorial methods are applicable to it.

In summary, our contributions are:
	\begin{enumerate}
		\item A theoretical analysis of the network flow lower bound showing its quality on uncongested networks, and
		\item A computational study on 16 networks of various size and levels of congestion, showing both that the network flow lower bound scales well computationally and that the quality of the bound is comparable to that of methods using heuristic bounds on the dual variables of DCOPF, even for congested networks where the theoretical results do not hold.
\end{enumerate}
\noindent In the remainder of this paper, we introduce the worst-case relay attack model in Section \ref{sec:formulation}, introduce the network flow restriction in Section~\ref{sec:nf-restriction}, state the main theoretical results related to it in Section \ref{sec:theoretical-results}, present a computational study demonstrating its efficacy in Section \ref{sec:computational-results}, and provide concluding thoughts in Section \ref{sec:conclusions}.

\section{Problem Formulation}\label{sec:formulation}

In this section, we introduce the notation we will use throughout the paper, as well as the worst-case relay attack model itself.

\subsection{Nomenclature}

We will use the following notation to describe the model.

\subsubsection{Sets}

\begin{description}
	\item [$\mathcal{K}$] Set of transmission lines
	\item [$\mathcal{G}$] Set of generators
	\item [$\mathcal{B}$] Set of buses
	\item [$\mathcal R$] Set of relays
	\item [$\mathcal{G}_b$] Set of generators at bus $b$
	\item [$\mathcal K^+_b$] Set of lines to bus $b$
	\item [$\mathcal K^-_b$] Set of lines originating at bus $b$
	\item [$\mathcal R_k$] Set of relays which control line $k$
	\item [$\mathcal R_g$] Set of relays which control generator $g$
	\item [$\mathcal R_b$] Set of relays which control bus $b$
\end{description}
Note that $\mathcal R_k$, $\mathcal R_g$, and $\mathcal R_b$ are not necessarily mutually disjoint. That is, a relay can control multiple grid components.

\subsubsection{Parameters}

\begin{description}
	\item [$B_k$] Susceptance of transmission line $k$
	\item [$\overline{F}_k$] Thermal limit for transmission line $k$
	\item [$o(k)$] Origin bus of transmission line $k$
	\item [$d(k)$] Destination bus of transmission line $k$
	\item [$b(g)$] Bus containing generator $g$
	\item  [$\overline{P}_g$] Upper limit of generator $g$ dispatch level
	\item [$D_b$] Demand at bus $b$
	\item [$U$] Attacker budget
\end{description}

\subsubsection{Variables}

\begin{description}
	\item [$\delta_r$] Indicator of whether or not relay $r$ is attacked
	\item [$u_g$] Indicator of whether or not generator $g$ is available
	\item [$v_k$] Indicator of whether or not line $k$ is available
	\item [$w_b$] Indicator of whether or not the load at bus $b$ can be served
	\item [$f_k$] Power flow through transmission line $k$
	\item [$p_g$] Generator dispatch level for generator $g$
	\item [$\theta_b$] Phase angle at bus $b$
	\item [$l_b$] Load shed at bus $b$
	\item [$\lambda_k^+$] Dual of line flow upper bound
	\item [$\lambda_k^-$] Dual of line flow lower bound
	\item [$\mu_b$] Dual of power balance constraint
	\item [$\alpha_b$] Dual of load shed lower bound
	\item [$\beta_b$] Dual of load shed upper bound
	\item [$\gamma_g$] Dual of generator dispatch upper bound
\end{description}

\subsection{Worst-Case Relay Attack Formulation}
The bilevel model is as follows:
\begin{subequations}\label{relay-attack-prob}
	\begin{align}
	\max_{\delta, u, v, w} \min_{f, p, l, \theta} \quad &\sum_{b \in \mathcal B} l_b \label{eq:ls-obj} \\
	\text{s.t.} \quad &\sum_{r \in \mathcal R} \delta_r \leq U \label{eq:budget} \\
	& \sum_{r \in \mathcal R_k} \delta_r \geq 1 - v_k & \forall k \in \mathcal K \quad & \phantom{(\gamma)} \label{eq:line-attacked} \\
	& \sum_{r \in \mathcal R_g} \delta_r \geq 1 - u_g & \forall g \in \mathcal G \quad & \phantom{(\gamma)} \label{eq:generator-attacked} \\
	& \sum_{r \in \mathcal R_b} \delta_r \geq 1 - w_b & \forall b \in \mathcal B \quad & \phantom{(\gamma)} \label{eq:bus-attacked}\\
	&\delta_r \leq 1 - v_k & \forall k \in \mathcal K, r \in \mathcal R_k \quad & \phantom{(\gamma)} \label{eq:relay-to-line}\\
	&\delta_r \leq 1 - u_g & \forall g \in \mathcal G, r \in \mathcal R_g \quad & \phantom{(\gamma)} \label{eq:relay-to-generator}\\
	&\delta_r \leq 1 - w_b & \forall b \in \mathcal B, r \in \mathcal R_b \quad & \phantom{(\gamma)} \label{eq:relay-to-bus}\\
	& \delta_r \in \{0, 1\} & \forall r \in \mathcal R \quad & \phantom{(\gamma)} \label{eq:delta-binary} \\
	& v_k \in \{0, 1\} & \forall k \in \mathcal K \quad & \phantom{(\gamma)} \label{eq:v-binary} \\
	& u_g \in \{0,1\} & \forall g \in \mathcal G \quad & \phantom{(\gamma)} \label{eq:u-binary}\\
	& w_b \in \{0,1\} & \forall b \in \mathcal B \quad & \phantom{(\gamma)} \label{eq:w-binary} \\
	&f_k \leq B_k(\theta_{o(k)} - \theta_{d(k)}) + 2\pi B_k(1 - v_k) & \forall k \in \mathcal K \quad & \phantom{(\xi^+)} \label{eq:ohms-nonlinear-ub}\\
	&f_k \geq B_k(\theta_{o(k)} - \theta_{d(k)}) - 2\pi B_k(1 - v_k) & \forall k \in \mathcal K \quad & \phantom{(\xi^-)} \label{eq:ohms-nonlinear-lb}\\
	&\sum_{k \in \mathcal K^+_b} f_k - \sum_{k \in \mathcal K^-_b} f_k + \sum_{g \in \mathcal G_b} p_g + l_b = D_b & \forall b \in \mathcal B \quad & (\mu) \label{eq:balance} \\
	& D_b(1 - w_b) \leq l_b \leq D_b & \forall b \in \mathcal B \quad & (\alpha, \beta) \label{eq:ls-bounds} \\
	&  -\overline F_kv_k \leq f_k \leq \overline F_kv_k & \forall k \in \mathcal K  \quad & (\lambda^+, \lambda^-) \label{eq:line-flow-bounds}\\
	&0 \leq p_g \leq \overline P_gu_g& \forall g \in \mathcal G \quad & (\gamma) \label{eq:generation-bounds} \\
	&l_b \geq 0 & \forall b \in \mathcal B \quad & \phantom{(\gamma)} \label{eq:ls-domain} \\
	&-\pi \leq \theta_b \leq \pi & \forall b \in \mathcal B  \quad & \phantom{(\kappa^+, \kappa^-)} \label{eq:angle-bounds}
	\end{align}
\end{subequations}
The attacker maximizes the total load shed and the defender minimizes it in (\ref{eq:ls-obj}).
Constraint (\ref{eq:budget}) ensures that the attacker does not exceed the cardinality budget for the number of relays he can attack.
Constraints (\ref{eq:line-attacked})-(\ref{eq:bus-attacked}) enforce that if a component is unavailable, a relay which connects to it must have been attacked.
The following three sets of constraints (\ref{eq:relay-to-line})-(\ref{eq:relay-to-bus}) enforce that if a relay connected to a line, generator, or load (respectively) is attacked, that component is unavailable to the defender.
Constraints (\ref{eq:delta-binary})-(\ref{eq:w-binary}) give the domain of the attacker's variables.
The defender's feasible region is defined by constraints (\ref{eq:ohms-nonlinear-ub})-(\ref{eq:angle-bounds}).
Constraints (\ref{eq:ohms-nonlinear-ub})-(\ref{eq:ohms-nonlinear-lb}) represent Ohm's law when $v_k = 1$ and are trivial when $v_k = 0$.
Constraint (\ref{eq:balance}) enforces power balance at each node.
Constraints (\ref{eq:ls-bounds})-(\ref{eq:angle-bounds}) enforce variable bounds and turn off components which are unavailable as a result of the attack.
Note that we assume that the generator dispatch lower bound is 0.
We do this to ensure that the defender problem is feasible for all attacks, since it is always feasible to generate no power and shed all the load.

\section{Network Flow Restriction}\label{sec:nf-restriction}

We propose solving a restriction of problem (\ref{relay-attack-prob}) in which we drop constraints (\ref{eq:ohms-nonlinear-ub}) and (\ref{eq:ohms-nonlinear-lb}), which consequently removes the phase angle variables $\theta$.
That is, we propose solving:
\begin{equation}\label{nf-restriction}
\begin{aligned}
\max_{\delta, u, v, w} \min_{f, p, l, \theta} \quad &\sum_{b \in \mathcal B} l_b \\
\text{s.t.} \quad & \text{(\ref{eq:budget})-(\ref{eq:w-binary}), (\ref{eq:balance})-(\ref{eq:ls-domain}).}
\end{aligned}
\end{equation}
Note that (\ref{nf-restriction}) gives a lower bound to problem (\ref{relay-attack-prob}) since it expands the feasible region of the defender, giving him more options to respond to the attack, and therefore decreasing the load shed from the attack.
We formulate a single-level bilinear reformulation of (\ref{nf-restriction}) by taking the dual of the defender problem:
\begin{subequations}\label{network-flow-single-level}
	\begin{align}
	\max_{\delta, u, v, w, \alpha, \beta, \lambda^+, \lambda^-, \gamma, \mu} \quad &\begin{aligned}& - \sum_{k \in \mathcal K} \overline F_k ( v_k \lambda_k^+ + v_k \lambda_k^-) \\
	&- \sum_{g \in \mathcal G} \overline P_g u_g \gamma_g + \sum_{b \in \mathcal B} D_b((1 - w_b) \alpha_b + \mu_b - \beta_b)
	\end{aligned} \label{nf-single-level-1}\\
	\text{s.t.} \quad &\text{(\ref{eq:budget})-(\ref{eq:w-binary})} \nonumber \\
	&\lambda_k^+  - \lambda_k^- + \mu_{d(k)} - \mu_{o(k)} = 0 \quad &\forall k \in \mathcal K \quad & (f) \label{nf:dual-flow} \\
	& \mu_{b(g)} - \gamma_g \leq 0 \quad &\forall g \in \mathcal G \quad & (p) \label{nf:dual-dispatch}\\
	& \alpha_b + \mu_b - \beta_b \leq 1 \quad &\forall b \in \mathcal B \quad & (l) \label{nf:dual-ls}\\
	& \alpha_b \geq 0 \quad &\forall b \in \mathcal B\quad & \phantom{(f)} \label{nf:dual-nonneg1} \\
	& \beta_b \geq 0 \quad &\forall b \in \mathcal B\quad & \phantom{(f)}\\
	& \lambda_k^+ \geq 0 \quad &\forall k \in \mathcal K \quad & \phantom{(f)}\\
	& \lambda_k^- \geq 0 \quad &\forall k \in \mathcal K \quad & \phantom{(f)}\\
	& \gamma_g \geq 0 \quad &\forall g \in \mathcal G.\quad & \phantom{(f)} \label{nf:dual-actually-last}
	\end{align}
\end{subequations}
Note that the objective function is bilinear, but since all the bilinear terms are products of a binary and a non-negative continuous variable, it is easily linearized if we have bounds on the continuous variables.
We can show that for this problem, 1 is a valid upper bound for all the dual variables.

\begin{observation}\label{obs:dual-bounds-are-1}
	In the formulation (\ref{nf-single-level-1})-(\ref{nf:dual-actually-last}), it is valid to add the constraints
	\begin{align*}
	&-1 \leq \mu_b \leq 1 & \forall b \in \mathcal B \\
	&\alpha_b \leq 1 & \forall b \in \mathcal B \\
	& \beta_b \leq 1 & \forall b \in \mathcal B \\
	& \lambda_k^+ \leq 1 &\forall k \in \mathcal K \\
	& \lambda_k^- \leq 1 & \forall k \in \mathcal K \\
	& \gamma_g \leq 1 & \forall g \in \mathcal G.
	\end{align*}
\end{observation}
The proof of Observation \ref{obs:dual-bounds-are-1} relies on both the fact that, after the constraints corresponding to Ohm's law are removed, the inner problem's constraint matrix is totally unimodular and the fact that all the coefficients of the objective are 1. Note that without the latter property, our results may not hold.
We give a formal proof of Observation \ref{obs:dual-bounds-are-1} in Appendix \ref{sec:proofs}, section \ref{sec:proof-dual-bounds-are-1}.
The application of the observation yields a mixed integer linear reformulation of (\ref{network-flow-single-level}):
\begin{subequations}\label{network-flow-single-level-linear}
	\begin{align}
	\max_{\delta, u, v, w, \alpha, \beta, \lambda^+, \lambda^-, \gamma, \mu, \overline \lambda^+, \overline \lambda^-, \overline \gamma, \overline \alpha} \quad &\begin{aligned}& - \sum_{k \in \mathcal K} \overline F_k ( \overline \lambda_k^+ + \overline \lambda_k^-) \\
	&- \sum_{g \in \mathcal G} \overline P_g \overline \gamma_g + \sum_{b \in \mathcal B} D_b(\overline  \alpha_b + \mu_b - \beta_b)
	\end{aligned} \\
	\text{s.t.} \quad &\text{(\ref{eq:budget})-(\ref{eq:w-binary}), (\ref{nf:dual-flow})-(\ref{nf:dual-actually-last})} \nonumber \\
	&\overline \lambda^+_k \leq v_k & \forall k \in \mathcal K \\
	& \lambda^+_k - 1 + v_k \leq \overline \lambda^+_k \leq \lambda_k^+ &\forall k \in \mathcal K \\
	&\overline \lambda^-_k \leq v_k & \forall k \in \mathcal K \\
	& \lambda^-_k - 1 + v_k \leq \overline \lambda^-_k \leq \lambda_k^-  &\forall k \in \mathcal K \\
	& \overline \gamma_g \leq u_g & \forall g \in \mathcal G \\
	& \gamma_g - 1 + u_g \leq \overline \gamma_g \leq \gamma_g  & \forall g \in \mathcal G \\
	& \overline \alpha_b \leq 1 - w_b & \forall b \in \mathcal B \\
	& \alpha_b - w_b \leq \overline \alpha_b \leq \alpha_b  & \forall g \in \mathcal G \\
	&\overline \lambda^+_k \geq 0 & \forall k \in \mathcal K \\
	&\overline \lambda^-_k \geq 0 & \forall k \in \mathcal K \\
	&\overline \gamma_g \geq 0 & \forall g \in \mathcal G \\
	&\overline \alpha_b \geq 0 & \forall b \in \mathcal B.
	\end{align}
\end{subequations}
Note that the optimal value of (\ref{network-flow-single-level-linear}) is a lower bound to that of (\ref{relay-attack-prob}).

Suppose $z^*$ is the optimal objective function value of (\ref{network-flow-single-level-linear}) and let $(\delta^*, u^*, v^*, w^*)$ be optimal solution of (\ref{network-flow-single-level-linear}) corresponding to the attacker's variables. 
Since we have removed the constraints corresponding to Ohm's law, it is possible that if we fix the attacker variables $(\delta, u, v, w)$ to $(\delta^*, u^*, v^*, w^*)$, the defender has no DCOPF-feasible solution.
In particular, this case implies that fixing $(\delta, u, v, w)$ to $(\delta^*, u^*, v^*, w^*)$ will lead to a higher load shed than $z^*$. Thus, in order to (i) obtain a feasible solution to our original problem (\ref{relay-attack-prob}) and (ii) possibly improve the bound obtained from (\ref{network-flow-single-level-linear}), we apply the steps described in Algorithm~\ref{algo:1}.

\begin{algorithm}
	\SetAlgoNoLine
	\DontPrintSemicolon
	\KwIn{All parameters of the power network, relays, and budget}
	\KwOut{A  high quality feasible solution to (\ref{relay-attack-prob}).}
	Solve (\ref{network-flow-single-level-linear}). Let $(\delta^*, u^*, v^*, w^*)$ be optimal solution of (\ref{network-flow-single-level-linear}) corresponding to the attacker's variables. \;
	Fix $(\delta, u, v, w)$ to $(\delta^*, u^*, v^*, w^*)$ and solve the resulting defender's DCOPF problem (with Ohm's law). \label{second-step} \;
	Return $(\delta^*, u^*, v^*, w^*)$, and the solution to the DCOPF problem.  \;
	\caption{Network Flow Lower Bound}
	\label{algo:1}
\end{algorithm}
In the remainder of the paper, we call the the resulting lower bound obtained from the load shed corresponding to the solution returned by Algorithm~\ref{algo:1} the \emph{network flow lower bound (NFLB)}.

In Section \ref{sec:computational-results}, we will show that empirically we find that  (\ref{network-flow-single-level-linear}) is efficiently solvable, even for large networks, and, as far as can be measured,
NFLB is a high-quality lower bound.
In the following section, we give some theoretical results which provide intuition for the good quality of NFLB.

\section{Theoretical Analysis of the Quality of the Network Flow Restriction}\label{sec:theoretical-results}

Though there are various notions of congestion in power networks, in this paper we describe a network as congested when the thermal limits on the transmission lines prevent a solution with less load shed.
Note that, in a DCOPF model, a bound on the phase angle difference can also be a source of congestion, but for a line $k$, this is only the case when 
\begin{equation}\label{angle-congestion}
2B_k \Theta < \overline F_k,
\end{equation}
where $\Theta$ is the phase angle difference bound.
That is, phase angles will become the limiting factor in how much power can be moved through the network if the maximum phase angle difference multiplied by the susceptance provides a tighter bound on the line flow than the thermal limit does.
However, in the case that (\ref{angle-congestion}) is true, we can replace the thermal limit with the left-hand side of (\ref{angle-congestion}) and drop phase angle difference bounds from the problem.
Thus, without loss of generality, in this paper we model DCOPF without phase angle difference bounds and consider congestion to be caused by restrictive thermal limits.

In the following, we show that, when the thermal limits are sufficiently large, it is always possible to find a DCOPF solution with the same injections as a network flow solution on the same network.
Throughout this section, we will assume that there is exactly one generator per bus.
Buses which do not have a generator can be represented as having a generator with maximum capacity 0, and buses with multiple generators can be represented as having one generator with capacity set to the sum of the capacities of the originals.
This is again because we assume that the minimum dispatch for a generator is always 0.
For notational convenience, we first introduce some definitions.

\begin{definition}
	Represent the network as a digraph $G(\mathcal B, \mathcal K)$ and let $N \in \{0, 1, -1\}^{|\mathcal B| \times |\mathcal K|}$ be the node-arc incidence matrix (where buses are nodes and lines are arcs).
	Let $\{d_b\}_{b \in \mathcal B}$ be a set of injections such that $\sum_{b \in \mathcal B} d_b = 0$\footnote{In terms of the notation used to describe (\ref{relay-attack-prob}), for a bus $b$, $d_b = \sum_{g \in \mathcal G b}p_g - (D_b - l_b)$, i.e., the power generated at the bus minus the load served at the bus.}.
	Then we say:
	\begin{itemize}
		\item The injection vector $d$ is {\it flow-polytope feasible} if there exists a vector of flows that satisfies thermal limits and flow conservation given the nodal injection values, i.e., $d$ is flow-polytope feasible if there exists $f \in \R^{|\mathcal K|}$ such that
		\[
		Nf = d, \text{ and } |f_k| \leq \overline F_k, \; \forall k \in \mathcal K.
		\]
		\item The injection vector $d$ is {\it DCOPF feasible} if there exists a flow vector that satisfies thermal limits, flow conservation given the nodal injection values, and Ohm's law. That is, $d$ is DCOPF feasible if there exists $f \in \R^{|\mathcal K|}$ such that 
		\[
		Nf = d, |f_k| \leq \overline F_k, \; \forall k \in \mathcal K, \text{ and } \exists \; \theta \in \R^{|\mathcal B|} \text{ such that } f_k = B_k(\theta_{o(k)} - \theta_{d(k)}), \; \forall k \in \mathcal K.
		\]
		Note that we do not require that $\theta$ satisfy the bounds given in (\ref{eq:angle-bounds}).
	\end{itemize}
\end{definition}
The set of DCOPF feasible injections is contained in the set of flow-polytope feasible injections.
We will show that the reverse is also true for uncongested networks, that is, when the thermal limits are sufficiently large.

\begin{definition}
	Given a connected digraph $G(\mathcal B, \mathcal K)$, consider a partition of the nodes formed by removing all the cut-arcs in the underlying graph and labeling the sets of nodes in each of the resulting connected components as $V^1, V^2, \dots, V^m$.
	Let
	\[
	r(G) := \max_{1 \leq i \leq m} |V^i|.
	\]
	Let $E(V^i)$ be the set of arcs with both end points in $V^i$. We will call the set of arcs $\cup_{i = 1}^m E(V^i)$ {\it non-cut-arcs}.
\end{definition}
Note that, in power networks, a partition of the nodes into more than one non-empty set is unusual since a cut-arc represents a single point of failure.
Thus we expect that for many of these networks, $r(G) = |\mathcal B|$.
Intuitively, the non-cut-arcs are the only ones for which the thermal limits could restrict our ability to find DCOPF feasible flows for a set of injections.
This is because, on a tree (and in the absence of phase angle bounds), there always exist phase angles such that a flow-polytope feasible flow is also DCOPF feasible.
Thus, the injections are certainly feasible.
Stated more simply, Ohm's law poses no additional restriction on flows if there are no cycles in the network.
Thus, our notion of ``large enough" thermal limits only applies to arcs which appear in cycles.
Formally, we have the following theorem, which we will prove in Appendix \ref{sec:proofs}, section \ref{sec:proof-of-thm}:
\begin{theorem}\label{thm:thermal-limit-bound}
	Consider a DCOPF problem on a connected digraph $G(\mathcal B, \mathcal K)$.
	Let $\mathcal B = V^1 \cup V^2 \cup \cdots \cup V^m$ such that $V^i \cap V^j = \emptyset$ for all $i \ne j$, and when we contract the nodes in $V^i$ into one node, the resulting graph is a tree.
	Let $r(G) := \max_{1 \leq i \leq m} |V^i|$. (Note that we can always select $r(G) = |\mathcal B|$.)
	Recall $\overline F_k$ is the thermal limit on arc $k$.
	Let $B_{\max}$ and $B_{\min}$ be the maximum and minimum susceptance respectively.
	If $d \in \R^{|\mathcal B|}$ is flow-polytope feasible and 
	\begin{equation}\label{eq:the-bound}
	\overline F_k \geq \sqrt{\frac{B_{\max}}{B_{\min}}} \frac{\sqrt{ r(G) - 1}}{2}\Vert d\Vert _1 \; \forall k \in \cup_{i = 1}^m E(V^i),
	\end{equation}
	then $d$ is also a DCOPF feasible injection.
\end{theorem}
Essentially, this means that, in uncongested networks, network flow is a good approximation for DCOPF when we consider the space of feasible injections.
More precisely:
\begin{corollary}\label{corollary:approximation-works}
	Consider the following problem:
	\begin{equation}\label{eq:dcopf}
	\begin{aligned}
	z^* = \min_{l, f, p, \theta} \quad &\sum_{b \in \mathcal B} l_b \\
	\text{s.t.} \quad & \sum_{k \in \mathcal K^+(b)} f_k - \sum_{k \in \mathcal K^-(b)} f_k + \sum_{g \in \mathcal G_b} p_g + l_b = D_b & \forall b \in \mathcal B \\
	&f_k = B_k(\theta_{o(k)} - \theta_{d(k)}) & \forall k \in \mathcal K \\
	&-\overline F_k \leq f_k \leq \overline F_k & \forall k \in \mathcal K \\
	&0 \leq l_b \leq D_b & \forall b \in \mathcal B \\
	& 0 \leq p_g  \leq  \overline P_g & \forall g \in \mathcal G \\
	\end{aligned}	
	\end{equation}
	and its relaxation
	\begin{equation}\label{eq:nf-relaxation}
	\begin{aligned}
	z^l = \min_{l, f, p, \theta} \quad &\sum_{b \in \mathcal B} l_b \\
	\text{s.t.} \quad & \sum_{k \in \mathcal K^+(b)} f_k - \sum_{k \in \mathcal K^-(b)} f_k + \sum_{g \in \mathcal G_b} p_g + l_b = D_b & \forall b \in \mathcal B \\
	&-\overline F_k \leq f_k \leq \overline F_k & \forall k \in \mathcal K \\
	&0 \leq l_b \leq D_b & \forall b \in \mathcal B \\
	& 0 \leq p_g  \leq  \overline P_g & \forall g \in \mathcal G
	\end{aligned}	
	\end{equation}
	Let $G(\mathcal B, \mathcal K)$ be the network digraph, $r(G)$ be as defined in Theorem \ref{thm:thermal-limit-bound}, and $D \in \R^{|\mathcal B|}$ be the vector of values $D_b$.
	Then if f $\overline F_k \geq \sqrt{\frac{B_{\max}}{B_{\min}}}\sqrt{r(G)-1}\Vert D\Vert _1$, we have $z^* = z^l$.
\end{corollary}
We provide a proof of this result in Appendix \ref{sec:proofs}, section \ref{sec:proof-approximation-works}.
It follows that, when the thermal limits respect the bound given in Corollary \ref{corollary:approximation-works} and we disregard phase angle bounds, the network flow relaxation of the worst-case relay attack problem is tight. 

\subsection{Tightness of the Bound from Theorem \ref{thm:thermal-limit-bound}}

As we will show through our empirical study in Section \ref{sec:computational-results}, we believe that in practice, the bound on the thermal limits in Theorem \ref{thm:thermal-limit-bound} is quite conservative.
That is, even for thermal limit values much smaller than the bound given in Theorem \ref{thm:thermal-limit-bound}, NFLB is the same optimal value as the original problem. 
However, we can show that for artificial instances the result is tight within a constant:
\begin{proposition}\label{prop:tight-within-constant}
	There exists a constant $c$ and a family of digraphs $\{G^n(V^n, E^n)\}_{n \in \N}$ where all susceptances are equal to 1 with corresponding injections $d^n \in \R^{|V^n|}$ and $\lim_{n \to \infty} \Vert d^n\Vert _1 = \infty$ such that if
	\[
	\overline F_e < c\sqrt{\frac{B_{\max}}{B_{\min}}} \frac{\sqrt{r(G)-1}}{2}\Vert d^n\Vert _1
	\]
	for all non-cut-arcs $e$ then $d^n$ is a flow-polytope feasible injection but not a DCOPF feasible injection.
\end{proposition}
A formal proof is provided in Appendix \ref{sec:proofs}, section \ref{sec:proof-tight-within-constant}, but an example of such a digraph is shown in Figure \ref{artifical-graph-n3}.
In essence, we can construct a digraph where the thermal limits satisfy the requirement of the theorem, but the triangles prevent there existing phase angles such that the lines can be used at capacity.
In practice, such an instance would be surprising, since most arcs in real networks are non-cut-arcs in order to prevent a single point of failure.

\begin{figure}
	\includegraphics[width=0.45\linewidth]{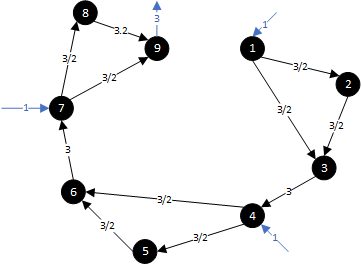}
	\hspace{0.25cm}
	\includegraphics[width=0.45\linewidth]{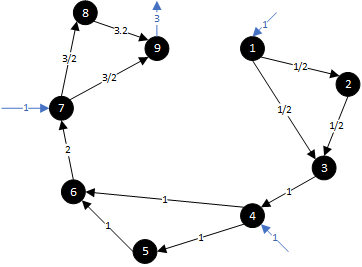}
	\caption{Example of artifical digraph used in the proof of Proposition \ref{prop:tight-within-constant} with $n = 3$.
		Black nodes and arcs represent network structure and blue arcs represent injections.
		The arcs are labeled with the thermal limits in the left-hand image and with flow-polytope feasible flows in the right-hand image.}
	\label{artifical-graph-n3}
\end{figure}

\section{Computational Results}\label{sec:computational-results}
\paragraph{The current state-of-the-art.} As discussed in the Introduction, the current state-of-the-art to obtain a lower bound for (\ref{relay-attack-prob}) is to solve single-level reformulations of (\ref{relay-attack-prob}). We present single-level reformulations of (\ref{relay-attack-prob}) in Appendix \ref{sec:single-level-reformulations}.
In the first formulation, problem (\ref{relay-attack-single-level-logical}), we use logical constraints to give an exact reformulation of (\ref{relay-attack-prob}) because we do not specify upper bounds on the dual variables.
It is possible to express this model using Gurobi's IndicatorConstraints.
While this problem is exact, we will show later in this section that the computational time to solve it is prohibitively large. 
We therefore also consider a mixed integer linear programming reformulation of (\ref{relay-attack-prob}), given in (\ref{relay-attack-single-level-linear}), in which upper bounds on the dual variables ($M$) are used to linearize the implications in (\ref{relay-attack-single-level-logical}). Again, as discussed before, we typically do not have good knowledge of these bounds $M$, and thus solving (\ref{relay-attack-single-level-logical}) with a heuristic value of $M$ leads to a lower bound.

\paragraph{Our Goal.} Since the network flow restriction also provides a lower bound for the worst-case relay attack problem (i.e. NFLB), we seek to answer two questions in this section:
\begin{enumerate}
	\item \emph{Computational Tractability}: What is the computational advantage of solving the network flow restriction over approaches that solve the single-level MILP obtained using arbitrary/heuristic bounds on the dual variables? \label{timeQ}
	\item \emph{Quality:} What is the quality of the network flow restriction solutions, i.e., that of NFLB? \label{qualityQ}
\end{enumerate}

\paragraph{Question 1.} In order to address Question~\ref{timeQ}, we compare the time it takes to obtain NFLB with the time it takes to find a solution with single-level reformulations of (\ref{relay-attack-prob}). In particular, after using Algorithm~\ref{algo:1}, we try to get a sense of the time it takes to find an equivalent-quality solution using the current state-of-the-art.  Therefore we proceed as follows:
\begin{itemize}
	\item We run Algorithm~\ref{algo:1}, which involves solving (\ref{network-flow-single-level-linear}) followed by solving an instance of DCOPF. We record the time needed to run Algorithm~\ref{algo:1}.
	\item We have two choices for single-level reformulations of (\ref{relay-attack-prob}): 
	\begin{itemize}
		\item Solve (\ref{relay-attack-single-level-logical}): We cut off the run when the lower bound is equal to the NFLB or terminate after 4 hours if we did not achieve the NFLB before then. 
		Unfortunately, while this method is the most attractive theoretically since problem (\ref{relay-attack-single-level-logical}) is an exact reformulation, we found that, even on the smallest test case, that it is much slower to solve than (\ref{relay-attack-single-level-linear}), taking over 4 hours to prove optimality on a 118-bus case with a 5\% attack budget.
		For this reason, we only report results with this model where we cut off at NFLB.
		\item Solve (\ref{relay-attack-single-level-linear}): We first need to decide $M$ values. 
		We set the value of $M$ in (\ref{relay-attack-single-level-linear}) to be the ceiling of the largest dual variable of the linear program solved in line \ref{second-step} of Algorithm~\ref{algo:1}. 
		For all of our test networks, applying the above procedure, we found $M = 2$ or $M = 3$. 
		It is possible we have discarded better solutions with this choice of $M$, but we at least ensure that we do not cut off the solution we have already found. 
		In addition, in experiments not reported here, we attempted larger values of $M$ for the smaller cases and found that the solution time scales badly with increases in $M$.
		In addition, we still did not find better solutions than the one corresponding to our chosen value of $M$.
		Again, we cut off the run of solving (\ref{relay-attack-single-level-linear}) when the lower bound is equal to the NFLB or terminate after 4 hours if we did not achieve the NFLB before then. We use Gurobi to solve (\ref{relay-attack-single-level-linear}), setting Gurobi's MIPFocus parameter to 1 to prioritize finding good quality solutions.
	\end{itemize}
\end{itemize}

\paragraph{Question 2.} In order to answer Question~\ref{qualityQ}, ideally, we should compare NFLB to the optimal solution of  (\ref{relay-attack-prob}).  
This is difficult to answer since we do not know of a non-trivial upper bound for the worst-case relay attack problem. In theory, we could accomplish this by solving (\ref{relay-attack-single-level-logical}). 
However, as mentioned earlier, we found that solving (\ref{relay-attack-single-level-logical}) using Gurobi's IndicatorConstraints does not scale well enough to beat NFLB.  
We therefore approach this question by comparing to the best lower bound obtained from the single-level reformulation of (\ref{relay-attack-single-level-linear}) with the heuristic choice of $M$ described above, warmstarted with the solution we found in Algorithm~\ref{algo:1}, and given a time limit of 4 hours.

\paragraph{Software and Hardware Specification.} Our models are implemented in Pyomo (\cite{pyomo-book, pyomo-paper}) using Gurobi 9.0.2 as the solver (\cite{gurobi}).
The experiments are run giving Gurobi 8 threads on a server with 40 Intel Xeon 2.20GHz CPUs and 251GB of RAM.

\subsection{Test Networks}

We present results on 16 different networks ranging from 118 buses to 6468 buses and with varying levels of congestion.
Details of the networks are given in Table \ref{table:test-instances}.
\begin{table}
	\centering
	\footnotesize
	\caption{Details of test instances used in the computational study. The last two columns give insight into congestion. We solve DCOPF with no attack and give the percentage of lines operating at their limits in the `Percentage of Thermal Limits Tight' column and the percentage of buses with a phase angle at its bound in the `Percentage of Phase Angle Bounds Tight' column. \label{table:test-instances}}
	\begin{tabular}{l | r r r r r r}
		\hline 
		Instance & \makecell{Number of \\Buses} & \makecell{Number of \\Lines} & \makecell{Number of \\Generators} & \makecell{Percentage of \\Thermal Limits Tight} & \makecell{Percentage of Phase \\Angle Bounds Tight} \\
		\hline
		118Blumsack & 118 & 186 & 19 &  0.54\% &  1.70\% \\
		300Kocuk &  300 & 411 & 61 &  2.92\% &  0.67\% \\
		500tamu & 500 & 597 & 90 &  0.00\% &  1.40\% \\
		1354pegase & 1354 & 1991 & 260 & 1.21\% & 0.74\% \\
		1354pegase\_api &  1354 &  1991 &  260 & 0.55\% & 1.11\% \\
		1354pegase\_sad &  1354 &  1991 &  260 & 1.31\% & 0.59\% \\
		1888rte & 1888 & 2531 & 297 & 0.91\% & 0.05\% \\
		 1888rte\_api &  1888 & 2531 & 297 & 2.96\% &  0.05\% \\
		 1888rte\_sad &  1888 & 2531 & 297 & 1.07\% & 0.05\% \\
		1951rte & 1951 & 2596 & 391 & 0.23\% & 0.05\% \\
		1951rte\_api & 1951 & 2596 & 391 & 4.08\% & 0.05\% \\
		1951rte\_sad & 1951 & 2596 & 391 & 0.58\% & 0.05\% \\
		2848rte & 2848 & 3776 & 547 &  0.66\% &  0.04\% \\
		3012wp & 3012 & 3572 & 502 &  0.03\% &  0.03\% \\
		3375wp & 3374 & 4161 & 596 &  0.26\% &  0.03\% \\
		6468rte & 6468 & 9000 & 1295 &  0.11\% &  0.02\% \\
		\hline
	\end{tabular}
\end{table}
Note that 118Blumsack is the IEEE 118 bus network as modified in \cite{BlumsackLI2007}.
This is a very congested network, which we use intentionally since congestion can break down the assumptions on dual bounds used in prior work and also renders our theoretical guarantees moot.
The 300Kocuk case is the IEEE 300 bus case, as modified in \cite{KocukJDLLS2016}.
It also has been modified to be more congested than the original.
The other cases are used as they are presented in \cite{pglib}.
Note that the cases with names ending in `\_api' and `\_sad' are congested modifications of the instance which shares the prefix of their name.

Since these test networks do not include any information about the control systems, for the sake of demonstration, we assume that there is one relay per bus which controls that bus, all the generators at that bus, and all the lines adjacent to the bus. 
For each of the test networks, we find a solution for the worst-case relay attack problem for attacker budgets of 1\%, 3\%, 5\%, 7\%, 10\%, 13\%, 15\%, 20\%, 25\%, and 30\% of the relays in the grid.

\subsection{Network Flow Restriction Results}

\paragraph{Difficulty in solving instance as a function of budget.} 
Both very small budgets and very large budgets turn out to be easier problems for Gurobi.
As is also observed in \cite{BienstockV2010}, we typically see that the computational time is longest for mid-range budgets.
This is intuitive since for small budgets there are fewer possible attacks, and for larger budgets, the attacker is able to shed all the load in the system, so the trade-offs are no longer interesting.
This concept is illustrated for a couple of the test networks in Figure \ref{fig:difficulty-and-saturation}.
In Figures \ref{sfig:118-bound} and \ref{sfig:1951-bound}, we see that as the attack budget increases, the amount of load shed achievable by the attacker increases and eventually saturates at the total load in the system.
Note that these plot the lower bound achieved at the 4-hour time limit, explaining why (\ref{relay-attack-single-level-logical}) can have a lower objective value than the other models.
In Figure \ref{sfig:118-comp-times}, we see that the most difficult problems computationally are at the elbow of the curve in Figure \ref{sfig:118-bound}.
Similarly in Figure \ref{sfig:1951-comp-times}, for a larger case, all the small budgets are difficult for Gurobi, but the problem becomes trivial after passing the saturation point.

\begin{figure}
	\centering
	\subfloat[Lower bounds for 118 Bus Case.]{
		\includegraphics[width=0.5\linewidth]{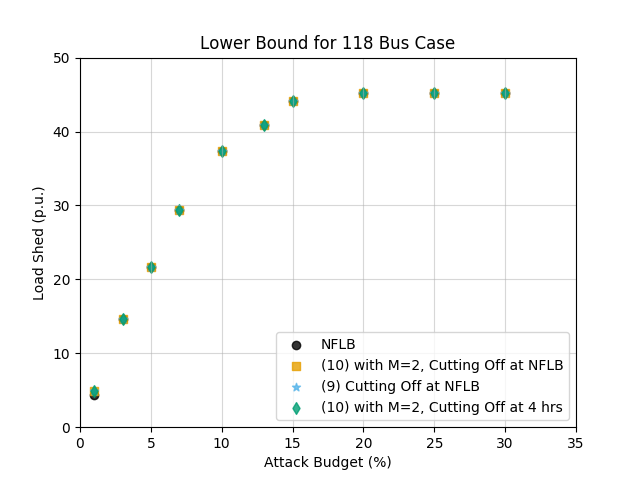}
		\label{sfig:118-bound}
	}
	\subfloat[Computational times for 118 Bus Case.]{
		\includegraphics[width=0.5\linewidth]{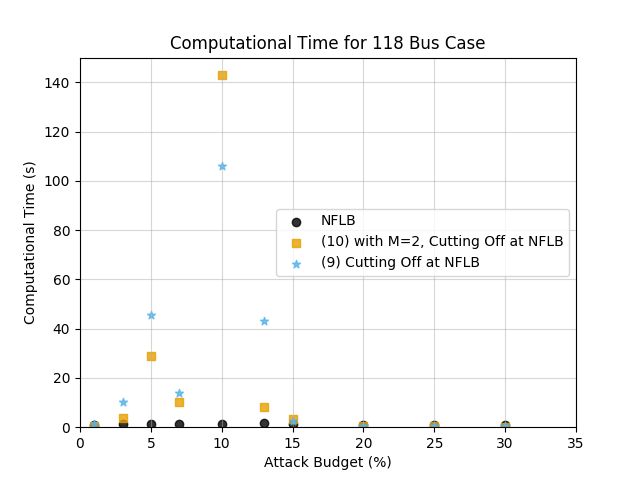}
		\label{sfig:118-comp-times}
	}
	
	\subfloat[Lower bounds for 1951 Bus Case.]{
		\includegraphics[width=0.5\linewidth]{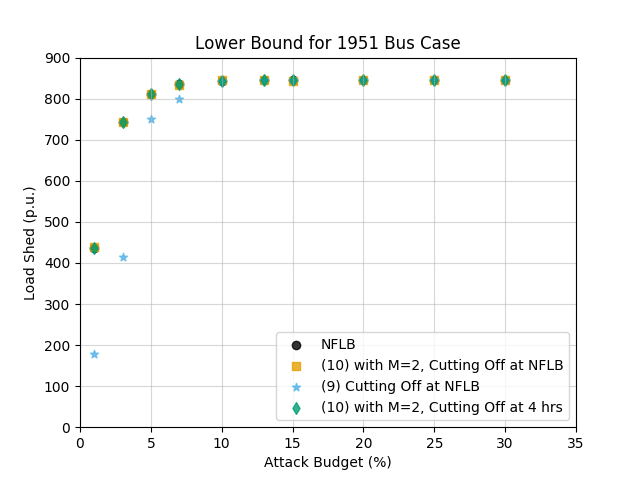}
		\label{sfig:1951-bound}
	}
	\subfloat[Computational times for 1951 Bus Case.]{
		\includegraphics[width=0.5\linewidth]{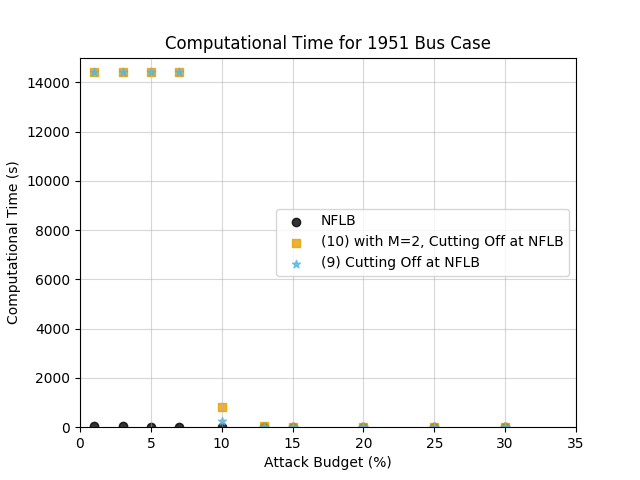}
		\label{sfig:1951-comp-times}
	}
	\caption{Computational times and bounds for different attack budgets on the 118 and 1951 bus cases. NFLB is obtained by running Algorithm~\ref{algo:1} using Gurobi. The other bounds are obtained using Gurobi to solve (\ref{relay-attack-single-level-logical}) and (\ref{relay-attack-single-level-linear}) terminating either when the bound reaches NFLB or after 4 hours. The numbers visualized in plots \ref{sfig:118-comp-times} and \ref{sfig:1951-comp-times} are given in the ``NFLB Time," ``Problem (\ref{relay-attack-single-level-linear}) Time to NFLB", and ``Problem (\ref{relay-attack-single-level-logical}) Time to NFLB" columns of Tables \ref{table:1} and \ref{table:2}.} \label{fig:difficulty-and-saturation}
\end{figure}
\paragraph{Results.}
The results from the network flow restriction and the experiments on problems (\ref{relay-attack-single-level-logical}) and (\ref{relay-attack-single-level-linear}) are shown in Tables \ref{table:1}, \ref{table:2}, \ref{table:3}, and \ref{table:4} for the twelve smaller test instances.
The first column shows the attack budget as a percentage of the relays in the system.
The second column translates this into an integer number of relays which can be attacked, that is, the value of $U$ for that instance.
The third column gives the best known lower bound from among all our experiments. 
This is the highest known load shed the attacker can achieve, given in per unit\footnote{As is typical in power systems modeling, in order to have better-scaled models, we quantify power ``per unit," that is, in 100 MW units.}.
In the ``NFLB" columns, ``Quality" is the load shed from the attack found by Algorithm~\ref{algo:1} as a percentage of the best known lower bound.
The time in seconds that it takes to run Algorithm~\ref{algo:1} is reported in the ``Time" column.
In the next three columns, we report results related to Question 2 above, that is, determining the quality of NFLB.
Recall that, in this experiment, we solve (\ref{relay-attack-single-level-linear}) using the heuristic $M$, warmstarting with the solution corresponding to NFLB, and allowing Gurobi a 4-hour time limit. 
The ``Problem (\ref{relay-attack-single-level-linear}) Quality" column gives the load shed this experiment achieved as a percentage of the best known lower bound.
The ``Problem (\ref{relay-attack-single-level-linear}) Time" column gives the time for the Gurobi solve.
The ``Problem (\ref{relay-attack-single-level-linear}) Gap" column reports the gap after the 4 hours.
Note that this is not a gap with a valid upper bound for the worst-case relay attack problem, but is instead a measure of how close Gurobi was to proving optimality on the particular restriction it was solving, in this case with the dual variables bounded by 2 for all but the 1888rte\_api and 1951rte\_api cases, where the dual variables are bounded by 3.
That is, Gurobi's upper bound is a bound on the best feasible solution achievable with this restriction.
In the last four columns, we report results related to Question 1 from the beginning of this section, in which we compare to solving the worst-case relay attack problem using formulations from prior literature.
In these experiments, we do not warmstart the Gurobi solves, and we cut off the solve when Gurobi achieves NFLB, if that is before the time limit of 4 hours.
We report the load shed achieved as a percentage of the best known lower bound as well as the time 
it takes Gurobi to find a solution whose objective value is as good as NFLB when solving (\ref{relay-attack-single-level-linear}) and (\ref{relay-attack-single-level-logical}) respectively.
As mentioned previously, we do not report results where we continue solving (\ref{relay-attack-single-level-logical}) after it achieves NFLB because we found it slow to find a solution as good as that obtained by Algorithm~\ref{algo:1}, even for the smaller test cases.

Gurobi hits the 4-hour time limit consistently for the more difficult budgets in the larger of these instances (i.e., Gurobi does not reach the NFLB within 4 hours).
Therefore we did not compare with solving either (\ref{relay-attack-single-level-logical}) or (\ref{relay-attack-single-level-linear}) for the four largest instances, and instead report the results of just the network flow restriction in Table \ref{table:large-nf}.
Column ``NFLB" gives the load shed from the attack found by Algorithm~\ref{algo:1} in per unit, and ``NFLB Time" gives the time taken to run Algorithm~\ref{algo:1}.

\begin{sidewaystable}
	\footnotesize
	\caption{Network flow restriction results on the six smallest cases, part 1. For each instance, we show results for 10 different budgets for the percentage of relays that can be attacked. The best known achievable load shed is in the ``Best Known LB" column. In the following two columns we give NFLB as a percentage of the best known solution for the instance and the computational time for Algorithm~\ref{algo:1}. The next three columns show the load shed attained by the solution we get from running (\ref{relay-attack-single-level-linear}) for up to 4 hours, the running time, and Gurobi's optimality gap at termination. The last four columns show the quality of the solution achieved and the times for Gurobi to achieve NFLB when solving problems (\ref{relay-attack-single-level-linear}) and (\ref{relay-attack-single-level-logical}) respectively. \label{table:1}}
	\begin{tabular}{c | r r r | r r | r r r | r r r r }
			\hline 
			\multirow{2}{1.6cm}{Instance} & \multirow{2}{1.3cm}{\makecell{Budget \\ (\%)}} & \multirow{2}{1.3cm}{\makecell{\# of \\ Relays}} & \multirow{2}{1.3cm}{\makecell{Best \\ Known \\ LB}} & \multicolumn{2}{c}{{\bf NFLB}} & \multicolumn{3}{c}{\bf{Question 2}} & \multicolumn{4}{c}{{\bf Question 1}} \\
			& & & & \makecell{Quality} & \makecell{Time \\ (s)} & \makecell{Problem \\ (\ref{relay-attack-single-level-linear}) \\ Quality } & \makecell{Problem \\ (\ref{relay-attack-single-level-linear}) \\ Time (s)} & \makecell{Problem \\ (\ref{relay-attack-single-level-linear}) Gap } &
			\makecell{Problem \\ (\ref{relay-attack-single-level-linear}) \\ Quality} & \makecell{Problem  \\ (\ref{relay-attack-single-level-linear}) Time \\ to NFLB (s)} & \makecell{Problem \\ (\ref{relay-attack-single-level-logical}) \\ Quality} & \makecell{Problem  \\ (\ref{relay-attack-single-level-logical}) Time \\ to NFLB (s)} \\
			\hline
			\multirow{10}{1.9cm}{118Blumsack}& 1 & 1 & 4.93 & 89.25\% & 1.02 & 100.00\% & 1.30 & 0.00\% & 100.00\% & 0.32 & 100.00\% & 1.17 \\
			& 3 & 4 & 14.63 & 100.00\% & 1.14 & 100.00\% & 64.97 & 0.00\% & 100.00\% & 3.70 & 100.00\% & 10.14 \\
			& 5 & 6 & 21.72 & 100.00\% & 1.42 & 100.00\% & 720.08 & 0.00\% & 100.00\% & 28.83 & 100.00\% & 45.58 \\
			& 7 & 8 & 29.38 & 100.00\% & 1.12 & 100.00\% & 387.48 & 0.00\% & 100.00\% & 10.24 & 100.00\% & 14.00 \\
			& 10 & 12 & 37.31 & 100.00\% & 1.25 & 100.00\% & 1294.79 & 0.00\% & 100.00\% & 142.95 & 100.00\% & 105.88 \\
			& 13 & 15 & 40.93 & 100.00\% & 1.71 & 100.00\% & 911.24 & 0.00\% & 100.00\% & 8.04 & 100.00\% & 43.11 \\
			& 15 & 18 & 44.19 & 100.00\% & 1.16 & 100.00\% & 102.94 & 0.00\% & 100.00\% & 3.27 & 100.00\% & 2.53 \\
			& 20 & 24 & 45.19 & 100.00\% & 0.80 & 100.00\% & 0.07 & 0.00\% & 100.00\% & 0.33 & 100.00\% & 0.35 \\
			& 25 & 30 & 45.19 & 100.00\% & 0.84 & 100.00\% & 0.06 & 0.00\% & 100.00\% & 0.29 & 100.00\% & 0.31 \\
			& 30 & 35 & 45.19 & 100.00\% & 0.80 & 100.00\% & 0.07 & 0.00\% & 100.00\% & 0.12 & 100.00\% & 0.29 \\
			\hline
			\multirow{10}{1.9cm}{300Kocuk}& 1 & 3 & 52.48 & 91.73\% & 2.26 & 100.00\% & 14.83 & 0.00\% & 100.00\% & 5.35 & 95.62\% & 15.98 \\
			& 3 & 9 & 99.55 & 99.77\% & 3.71 & 100.00\% & 14400.01 & 22.75\% & 100.00\% & 99.53 & 100.00\% & 3065.81 \\
			& 5 & 15 & 130.24 & 100.00\% & 6.47 & 100.00\% & 14400.02 & 27.94\% & 100.00\% & 385.57 & 89.58\% & 14400.00 \\
			& 7 & 21 & 152.01 & 100.00\% & 8.84 & 100.00\% & 14400.01 & 28.92\% & 100.00\% & 538.82 & 97.95\% & 14400.01 \\
			& 10 & 30 & 184.07 & 100.00\% & 5.90 & 100.00\% & 14400.01 & 19.99\% & 100.00\% & 199.14 & 98.35\% & 14400.00 \\
			& 13 & 39 & 211.37 & 100.00\% & 3.62 & 100.00\% & 14400.01 & 9.16\% & 100.00\% & 355.72 & 99.71\% & 14400.01 \\
			& 15 & 45 & 224.88 & 100.00\% & 3.18 & 100.00\% & 14400.01 & 4.27\% & 100.00\% & 212.62 & 99.63\% & 14400.00 \\
			& 20 & 60 & 238.48 & 100.00\% & 2.35 & 100.00\% & 0.13 & 0.00\% & 100.00\% & 19.44 & 100.00\% & 4.10 \\
			& 25 & 75 & 238.48 & 100.00\% & 2.07 & 100.00\% & 0.10 & 0.00\% & 100.00\% & 2.32 & 100.00\% & 1.72 \\
			& 30 & 90 & 238.48 & 100.00\% & 1.59 & 100.00\% & 0.10 & 0.00\% & 100.00\% & 1.19 & 100.00\% & 2.74 \\
			\hline
			\multirow{10}{1.9cm}{500tamu}& 1 & 5 & 16.79 & 100.00\% & 13.71 & 100.00\% & 12764.73 & 0.00\% & 100.00\% & 3872.70 & 88.62\% & 14400.01 \\
			& 3 & 15 & 71.88 & 100.00\% & 3.92 & 100.00\% & 14400.02 & 7.42\% & 100.00\% & 44.53 & 100.00\% & 82.53 \\
			& 5 & 25 & 77.26 & 100.00\% & 2.88 & 100.00\% & 14400.02 & 0.32\% & 100.00\% & 44.11 & 100.00\% & 81.06 \\
			& 7 & 35 & 77.51 & 100.00\% & 2.87 & 100.00\% & 0.15 & 0.00\% & 100.00\% & 14.38 & 100.00\% & 20.61 \\
			& 10 & 50 & 77.51 & 100.00\% & 2.17 & 100.00\% & 0.15 & 0.00\% & 100.00\% & 2.82 & 100.00\% & 4.97 \\
			& 13 & 65 & 77.51 & 100.00\% & 2.25 & 100.00\% & 0.15 & 0.00\% & 100.00\% & 2.16 & 100.00\% & 1.73 \\
			& 15 & 75 & 77.51 & 100.00\% & 2.89 & 100.00\% & 0.15 & 0.00\% & 100.00\% & 1.02 & 100.00\% & 1.74 \\
			& 20 & 100 & 77.51 & 100.00\% & 2.08 & 100.00\% & 0.15 & 0.00\% & 100.00\% & 3.09 & 100.00\% & 1.90 \\
			& 25 & 125 & 77.51 & 100.00\% & 2.15 & 100.00\% & 0.15 & 0.00\% & 100.00\% & 1.66 & 100.00\% & 1.59 \\
			& 30 & 150 & 77.51 & 100.00\% & 2.21 & 100.00\% & 0.15 & 0.00\% & 100.00\% & 0.99 & 100.00\% & 1.41 \\
			\hline
	\end{tabular}
\end{sidewaystable}

\begin{sidewaystable}
	\footnotesize
	\caption{Network flow restriction results on the six smallest cases, part 2. For each instance, we show results for 10 different budgets for the percentage of relays that can be attacked. The best known achievable load shed is in the ``Best Known LB" column. In the following two columns we give NFLB as a percentage of the best known solution for the instance and the computational time for Algorithm~\ref{algo:1}. The next three columns show the load shed attained by the solution we get from running (\ref{relay-attack-single-level-linear}) for up to 4 hours, the running time, and Gurobi's optimality gap at termination. The last four columns show the quality of the solution achieved and the times for Gurobi to achieve NFLB when solving problems (\ref{relay-attack-single-level-linear}) and (\ref{relay-attack-single-level-logical}) respectively. Note that, in the Question 1 results, in cases where problem (\ref{relay-attack-single-level-linear}) runs for 4 hours but has a quality of 100.00\%, this is a symptom of rounding: The NFLB is not quite achieved within the time limit, but that is not reflected within the two decimal places in this table.
		Also note that the Question 1 experiment solving problem (\ref{relay-attack-single-level-linear}) occasionally finds the best known solution since it can exceed the NFLB in the iteration before it terminates.\label{table:2}}
	\begin{tabular}{c | r r r | r  r | r r  r | r r r r }
			\hline 
			\multirow{2}{1.5cm}{Instance} & \multirow{2}{1.3cm}{\makecell{Budget \\ (\%)}} & \multirow{2}{1.3cm}{\makecell{\# of \\ Relays}} & \multirow{2}{1.3cm}{\makecell{Best \\ Known \\ LB}} & \multicolumn{2}{c}{{\bf NFLB}} & \multicolumn{3}{c}{\bf{Question 2}} & \multicolumn{4}{c}{{\bf Question 1}} \\
			& & & & \makecell{Quality} & \makecell{Time \\ (s)} & \makecell{Problem \\ (\ref{relay-attack-single-level-linear}) \\ Quality } & \makecell{Problem \\ (\ref{relay-attack-single-level-linear}) \\ Time (s)} & \makecell{Problem \\ (\ref{relay-attack-single-level-linear}) Gap } &
			\makecell{Problem \\ (\ref{relay-attack-single-level-linear}) \\ Quality} & \makecell{Problem  \\ (\ref{relay-attack-single-level-linear}) Time \\ to NFLB (s)} & \makecell{Problem \\ (\ref{relay-attack-single-level-logical}) \\ Quality} & \makecell{Problem  \\ (\ref{relay-attack-single-level-logical}) Time \\ to NFLB (s)} \\
			\hline
			\multirow{10}{1.9cm}{1354pegase}& 1 & 14 & 231.67 & 100.00\% & 139.59 & 100.00\% & 14400.04 & 135.91\% & 97.74\% & 14400.02 & 47.40\% & 14400.01 \\
			& 3 & 41 & 532.21 & 100.00\% & 44.77 & 100.00\% & 14400.03 & 30.74\% & 98.91\% & 14400.05 & 68.91\% & 14400.01 \\
			& 5 & 68 & 653.14 & 100.00\% & 23.55 & 100.00\% & 14400.05 & 9.19\% & 99.18\% & 14400.06 & 93.28\% & 14400.01 \\
			& 7 & 95 & 698.13 & 100.00\% & 18.53 & 100.00\% & 14400.18 & 6.21\% & 99.87\% & 14400.26 & 98.21\% & 14400.00 \\
			& 10 & 135 & 735.70 & 100.00\% & 20.94 & 100.00\% & 14400.02 & 0.78\% & 100.00\% & 14400.03 & 99.69\% & 14400.01 \\
			& 13 & 176 & 741.46 & 100.00\% & 12.55 & 100.00\% & 6.48 & 0.00\% & 100.00\% & 179.52 & 100.00\% & 77.53 \\
			& 15 & 203 & 741.46 & 100.00\% & 9.30 & 100.00\% & 0.45 & 0.00\% & 100.00\% & 85.14 & 100.00\% & 64.13 \\
			& 20 & 271 & 741.46 & 100.00\% & 7.03 & 100.00\% & 0.46 & 0.00\% & 99.99\% & 15.57 & 100.00\% & 12.61 \\
			& 25 & 338 & 741.46 & 100.00\% & 6.23 & 100.00\% & 0.47 & 0.00\% & 100.00\% & 9.65 & 100.00\% & 8.50 \\
			& 30 & 406 & 741.46 & 100.00\% & 6.10 & 100.00\% & 0.47 & 0.00\% & 100.00\% & 5.09 & 100.00\% & 8.91 \\
			\hline
			\multirow{10}{1.9cm}{ 1888rte}& 1 & 19 & 292.08 & 100.00\% & 45.70 & 100.00\% & 14400.03 & 76.57\% & 100.00\% & 8629.13 & 36.04\% & 14400.01 \\
			& 3 & 57 & 525.44 & 100.00\% & 28.92 & 100.00\% & 14400.02 & 8.02\% & 100.00\% & 3790.54 & 85.52\% & 14400.05 \\
			& 5 & 94 & 572.74 & 100.00\% & 27.58 & 100.00\% & 14400.17 & 2.45\% & 99.97\% & 14400.18 & 70.71\% & 14400.01 \\
			& 7 & 132 & 591.76 & 100.00\% & 21.69 & 100.00\% & 14400.06 & 0.73\% & 99.99\% & 14400.07 & 97.80\% & 14400.02 \\
			& 10 & 189 & 596.07 & 100.00\% & 12.60 & 100.00\% & 0.74 & 0.00\% & 100.00\% & 360.94 & 100.00\% & 28.14 \\
			& 13 & 245 & 596.07 & 100.00\% & 9.42 & 100.00\% & 0.77 & 0.00\% & 100.00\% & 43.09 & 100.00\% & 23.44 \\
			& 15 & 283 & 596.07 & 100.00\% & 8.83 & 100.00\% & 0.76 & 0.00\% & 100.00\% & 15.65 & 100.00\% & 16.64 \\
			& 20 & 378 & 596.07 & 100.00\% & 8.38 & 100.00\% & 0.72 & 0.00\% & 100.00\% & 5.70 & 100.00\% & 14.25 \\
			& 25 & 472 & 596.07 & 100.00\% & 8.12 & 100.00\% & 0.74 & 0.00\% & 100.00\% & 14.02 & 99.99\% & 8.88 \\
			& 30 & 566 & 596.07 & 100.00\% & 8.06 & 100.00\% & 0.73 & 0.00\% & 100.00\% & 6.55 & 99.99\% & 9.77 \\
			\hline
			\multirow{10}{1.9cm}{1951rte}& 1 & 20 & 441.07 & 99.06\% & 42.70 & 99.97\% & 14400.02 & 65.13\% & 100.00\% & 14400.03 & 40.21\% & 14400.01 \\
			& 3 & 59 & 743.17 & 100.00\% & 35.71 & 100.00\% & 14400.04 & 9.32\% & 99.89\% & 14400.04 & 55.82\% & 14400.03 \\
			& 5 & 98 & 810.79 & 100.00\% & 26.35 & 100.00\% & 14400.20 & 2.51\% & 100.00\% & 14400.04 & 92.46\% & 14400.01 \\
			& 7 & 137 & 834.51 & 100.00\% & 26.13 & 100.00\% & 14400.02 & 0.97\% & 100.00\% & 14400.06 & 95.85\% & 14400.01 \\
			& 10 & 195 & 844.23 & 100.00\% & 18.37 & 100.00\% & 14.83 & 0.00\% & 100.00\% & 972.18 & 100.00\% & 255.64 \\
			& 13 & 254 & 844.27 & 100.00\% & 13.43 & 100.00\% & 2.51 & 0.00\% & 100.00\% & 47.39 & 100.00\% & 15.35 \\
			& 15 & 293 & 844.27 & 100.00\% & 11.70 & 100.00\% & 0.80 & 0.00\% & 100.00\% & 44.41 & 100.00\% & 14.70 \\
			& 20 & 390 & 844.27 & 100.00\% & 11.21 & 100.00\% & 0.91 & 0.00\% & 100.00\% & 6.63 & 100.00\% & 14.59 \\
			& 25 & 488 & 844.27 & 100.00\% & 11.19 & 100.00\% & 3.26 & 0.00\% & 100.00\% & 10.19 & 100.00\% & 11.55 \\
			& 30 & 585 & 844.27 & 100.00\% & 11.09 & 100.00\% & 2.16 & 0.00\% & 100.00\% & 9.46 & 100.00\% & 12.47 \\
			\hline
\end{tabular}
\end{sidewaystable}

\begin{sidewaystable}
	\footnotesize
	\caption{Network flow restriction results on the `api' congested cases. For each instance, we show results for 10 different budgets for the percentage of relays that can be attacked. The best known achievable load shed is in the ``Best Known LB" column. In the following two columns we give NFLB as a percentage of the best known solution for the instance and the computational time for Algorithm~\ref{algo:1}. The next three columns show the load shed attained by the solution we get from running (\ref{relay-attack-single-level-linear}) for up to 4 hours, the running time, and Gurobi's optimality gap at termination. The last four columns show the quality of the solution achieved and the times for Gurobi to achieve NFLB when solving problems (\ref{relay-attack-single-level-linear}) and (\ref{relay-attack-single-level-logical}) respectively. \label{table:3}}
	\begin{tabular}{c | r r r | r r | r r r | r r r r }
		\hline 
		\multirow{2}{2.6cm}{Instance} & \multirow{2}{0.9cm}{\makecell{Budget \\ (\%)}} & \multirow{2}{0.9cm}{\makecell{\# of \\ Relays}} & \multirow{2}{1.1cm}{\makecell{Best \\ Known \\ LB}} & \multicolumn{2}{c}{{\bf NFLB}} & \multicolumn{3}{c}{\bf{Question 2}} & \multicolumn{4}{c}{{\bf Question 1}} \\
		& & & & \makecell{Quality} & \makecell{Time \\ (s)} & \makecell{Problem \\ (\ref{relay-attack-single-level-linear}) \\ Quality } & \makecell{Problem \\ (\ref{relay-attack-single-level-linear}) \\ Time (s)} & \makecell{Problem \\ (\ref{relay-attack-single-level-linear}) Gap } &
		\makecell{Problem \\ (\ref{relay-attack-single-level-linear}) \\ Quality} & \makecell{Problem  \\ (\ref{relay-attack-single-level-linear}) Time \\ to NFLB (s)} & \makecell{Problem \\ (\ref{relay-attack-single-level-logical}) \\ Quality} & \makecell{Problem  \\ (\ref{relay-attack-single-level-logical}) Time \\ to NFLB (s)} \\
		\hline
		\multirow{10}{1.9cm}{1354pegase\_api}& 1 & 14 & 223.18 & 100.00\% & 60.56 & 100.00\% & 14400.03 & 109.28\% & 100.00\% & 14400.04 & 48.75\% & 14400.01 \\
		& 3 & 41 & 471.12 & 100.00\% & 34.50 & 100.00\% & 14400.05 & 56.92\% & 100.00\% & 14400.58 & 69.15\% & 14400.01 \\
		& 5 & 68 & 631.65 & 99.88\% & 35.31 & 100.00\% & 14400.02 & 23.98\% & 99.83\% & 14400.02 & 89.42\% & 14400.10 \\
		& 7 & 95 & 735.20 & 100.00\% & 29.10 & 100.00\% & 14400.20 & 10.53\% & 99.85\% & 14400.06 & 96.02\% & 14400.01 \\
		& 10 & 135 & 808.01 & 100.00\% & 23.11 & 100.00\% & 14400.04 & 0.57\% & 100.00\% & 14400.04 & 99.44\% & 14400.01 \\
		& 13 & 176 & 812.59 & 100.00\% & 11.79 & 100.00\% & 0.46 & 0.00\% & 100.00\% & 186.47 & 100.00\% & 190.42 \\
		& 15 & 203 & 812.59 & 100.00\% & 12.90 & 100.00\% & 0.47 & 0.00\% & 99.99\% & 105.22 & 100.00\% & 68.38 \\
		& 20 & 271 & 812.59 & 100.00\% & 8.92 & 100.00\% & 0.47 & 0.00\% & 100.00\% & 21.78 & 100.00\% & 9.08 \\
		& 25 & 338 & 812.59 & 100.00\% & 8.42 & 100.00\% & 0.47 & 0.00\% & 100.00\% & 11.50 & 100.00\% & 7.02 \\
		& 30 & 406 & 812.59 & 100.00\% & 8.32 & 100.00\% & 0.47 & 0.00\% & 100.00\% & 5.82 & 100.00\% & 6.98 \\
		\hline
		\multirow{10}{1.9cm}{1888rte\_api}& 1 & 19 & 310.43 & 99.31\% & 106.29 & 100.00\% & 14400.02 & 122.91\% & 96.86\% & 14400.03 & 47.88\% & 14400.02 \\
		& 3 & 57 & 628.63 & 100.00\% & 30.63 & 100.00\% & 14400.05 & 22.79\% & 99.49\% & 14400.01 & 68.16\% & 14400.01 \\
		& 5 & 94 & 755.83 & 100.00\% & 28.23 & 100.00\% & 14400.03 & 4.93\% & 99.47\% & 14400.04 & 93.57\% & 14400.01 \\
		& 7 & 132 & 796.28 & 100.00\% & 23.25 & 100.00\% & 14400.06 & 1.04\% & 99.74\% & 14400.06 & 99.26\% & 14400.02 \\
		& 10 & 189 & 804.53 & 100.00\% & 12.00 & 100.00\% & 0.75 & 0.00\% & 100.00\% & 641.89 & 100.00\% & 31.14 \\
		& 13 & 245 & 804.53 & 100.00\% & 8.73 & 100.00\% & 0.76 & 0.00\% & 100.00\% & 119.60 & 100.00\% & 10.60 \\
		& 15 & 283 & 804.53 & 100.00\% & 8.31 & 100.00\% & 0.74 & 0.00\% & 100.00\% & 83.06 & 100.00\% & 10.24 \\
		& 20 & 378 & 804.53 & 100.00\% & 8.24 & 100.00\% & 0.73 & 0.00\% & 100.00\% & 27.09 & 100.00\% & 9.41 \\
		& 25 & 472 & 804.53 & 100.00\% & 8.30 & 100.00\% & 0.74 & 0.00\% & 100.00\% & 25.81 & 100.00\% & 9.08 \\
		& 30 & 566 & 804.53 & 100.00\% & 8.02 & 100.00\% & 0.73 & 0.00\% & 100.00\% & 5.54 & 100.00\% & 9.62 \\
		\hline
		\multirow{10}{1.9cm}{1951rte\_api}& 1 & 20 & 418.14 & 100.00\% & 109.68 & 100.00\% & 14400.11 & 106.90\% & 98.00\% & 14400.02 & 54.80\% & 14400.02 \\
		& 3 & 59 & 777.21 & 100.00\% & 42.68 & 100.00\% & 14400.04 & 23.11\% & 99.85\% & 14400.02 & 76.25\% & 14400.01 \\
		& 5 & 98 & 935.91 & 100.00\% & 32.71 & 100.00\% & 14400.06 & 5.01\% & 99.71\% & 14400.33 & 90.71\% & 14400.01 \\
		& 7 & 137 & 978.88 & 100.00\% & 28.69 & 100.00\% & 14400.07 & 1.22\% & 99.58\% & 14400.06 & 88.29\% & 14400.02 \\
		& 10 & 195 & 993.09 & 100.00\% & 18.06 & 100.00\% & 17.78 & 0.01\% & 100.00\% & 792.13 & 100.00\% & 228.47 \\
		& 13 & 254 & 993.11 & 100.00\% & 14.00 & 100.00\% & 0.76 & 0.00\% & 100.00\% & 341.08 & 100.00\% & 17.00 \\
		& 15 & 293 & 993.11 & 100.00\% & 13.84 & 100.00\% & 3.26 & 0.00\% & 100.00\% & 108.54 & 100.00\% & 11.90 \\
		& 20 & 390 & 993.11 & 100.00\% & 11.57 & 100.00\% & 1.65 & 0.00\% & 99.99\% & 13.96 & 100.00\% & 11.69 \\
		& 25 & 488 & 993.11 & 100.00\% & 11.26 & 100.00\% & 0.79 & 0.00\% & 100.00\% & 12.06 & 100.00\% & 11.68 \\
		& 30 & 585 & 993.11 & 100.00\% & 11.07 & 100.00\% & 0.83 & 0.00\% & 100.00\% & 5.61 & 100.00\% & 10.13 \\
		\hline
	\end{tabular}
\end{sidewaystable}

\begin{sidewaystable}
	\footnotesize
	\caption{Network flow restriction results on the `sad' congested cases. For each instance, we show results for 10 different budgets for the percentage of relays that can be attacked. The best known achievable load shed is in the ``Best Known LB" column. In the following two columns we give NFLB as a percentage of the best known solution for the instance and the computational time for Algorithm~\ref{algo:1}. The next three columns show the load shed attained by the solution we get from running (\ref{relay-attack-single-level-linear}) for up to 4 hours, the running time, and Gurobi's optimality gap at termination. The last four columns show the quality of the solution achieved and the times for Gurobi to achieve NFLB when solving problems (\ref{relay-attack-single-level-linear}) and (\ref{relay-attack-single-level-logical}) respectively. \label{table:4}}
	\begin{tabular}{c | r r r | r r | r r r | r r r r }
		\hline 
		\multirow{2}{2.6cm}{Instance} & \multirow{2}{0.9cm}{\makecell{Budget \\ (\%)}} & \multirow{2}{0.9cm}{\makecell{\# of \\ Relays}} & \multirow{2}{1.1cm}{\makecell{Best \\ Known \\ LB}} & \multicolumn{2}{c}{{\bf NFLB}} & \multicolumn{3}{c}{\bf{Question 2}} & \multicolumn{4}{c}{{\bf Question 1}} \\
		& & & & \makecell{Quality} & \makecell{Time \\ (s)} & \makecell{Problem \\ (\ref{relay-attack-single-level-linear}) \\ Quality } & \makecell{Problem \\ (\ref{relay-attack-single-level-linear}) \\ Time (s)} & \makecell{Problem \\ (\ref{relay-attack-single-level-linear}) Gap } &
		\makecell{Problem \\ (\ref{relay-attack-single-level-linear}) \\ Quality} & \makecell{Problem  \\ (\ref{relay-attack-single-level-linear}) Time \\ to NFLB (s)} & \makecell{Problem \\ (\ref{relay-attack-single-level-logical}) \\ Quality} & \makecell{Problem  \\ (\ref{relay-attack-single-level-logical}) Time \\ to NFLB (s)} \\
		\hline
		\multirow{10}{1.9cm}{1354pegase\_sad}& 1 & 14 & 237.57 & 97.60\% & 230.54 & 97.60\% & 14400.03 & 136.01\% & 100.00\% & 14400.03 & 41.14\% & 14400.02 \\
		& 3 & 41 & 533.47 & 100.00\% & 47.20 & 100.00\% & 14400.06 & 29.97\% & 98.98\% & 14400.01 & 63.64\% & 14400.01 \\
		& 5 & 68 & 653.14 & 100.00\% & 22.59 & 100.00\% & 14400.04 & 9.28\% & 99.64\% & 14400.36 & 89.63\% & 14400.01 \\
		& 7 & 95 & 698.13 & 100.00\% & 21.08 & 100.00\% & 14400.06 & 6.21\% & 99.88\% & 14400.03 & 98.69\% & 14400.01 \\
		& 10 & 135 & 735.70 & 100.00\% & 21.86 & 100.00\% & 14400.03 & 0.78\% & 100.00\% & 14400.03 & 99.94\% & 14400.02 \\
		& 13 & 176 & 741.46 & 100.00\% & 13.25 & 100.00\% & 0.46 & 0.00\% & 100.00\% & 173.48 & 100.00\% & 62.05 \\
		& 15 & 203 & 741.46 & 100.00\% & 11.88 & 100.00\% & 0.45 & 0.00\% & 99.99\% & 98.99 & 100.00\% & 8.51 \\
		& 20 & 271 & 741.46 & 100.00\% & 10.23 & 100.00\% & 0.46 & 0.00\% & 100.00\% & 20.86 & 100.00\% & 13.63 \\
		& 25 & 338 & 741.46 & 100.00\% & 8.62 & 100.00\% & 0.46 & 0.00\% & 100.00\% & 8.16 & 100.00\% & 9.98 \\
		& 30 & 406 & 741.46 & 100.00\% & 8.71 & 100.00\% & 0.45 & 0.00\% & 100.00\% & 3.14 & 100.00\% & 11.21 \\
		\hline
		\multirow{10}{1.9cm}{1888rte\_sad}& 1 & 19 & 301.49 & 100.00\% & 53.02 & 100.00\% & 14400.09 & 69.67\% & 100.00\% & 14400.02 & 32.14\% & 14400.03 \\
		& 3 & 57 & 525.74 & 100.00\% & 28.91 & 100.00\% & 14400.03 & 8.31\% & 99.94\% & 14400.03 & 84.64\% & 14400.02 \\
		& 5 & 94 & 572.77 & 100.00\% & 24.08 & 100.00\% & 14400.20 & 2.61\% & 99.99\% & 14400.05 & 96.22\% & 14400.03 \\
		& 7 & 132 & 591.76 & 100.00\% & 20.97 & 100.00\% & 14400.24 & 0.73\% & 99.99\% & 14400.04 & 98.53\% & 14400.02 \\
		& 10 & 189 & 596.07 & 100.00\% & 11.95 & 100.00\% & 3.40 & 0.00\% & 100.00\% & 435.82 & 100.00\% & 43.82 \\
		& 13 & 245 & 596.07 & 100.00\% & 8.38 & 100.00\% & 2.86 & 0.00\% & 100.00\% & 49.02 & 100.00\% & 11.23 \\
		& 15 & 283 & 596.07 & 100.00\% & 8.65 & 100.00\% & 2.21 & 0.00\% & 100.00\% & 23.46 & 100.00\% & 15.20 \\
		& 20 & 378 & 596.07 & 100.00\% & 8.43 & 100.00\% & 3.22 & 0.00\% & 100.00\% & 21.50 & 100.00\% & 10.79 \\
		& 25 & 472 & 596.07 & 100.00\% & 8.06 & 100.00\% & 3.21 & 0.00\% & 100.00\% & 6.15 & 99.99\% & 8.61 \\
		& 30 & 566 & 596.07 & 100.00\% & 8.14 & 100.00\% & 2.39 & 0.00\% & 100.00\% & 2.94 & 99.99\% & 9.31 \\
		\hline
		\multirow{10}{1.9cm}{1951rte\_sad}& 1 & 20 & 451.75 & 100.00\% & 44.28 & 100.00\% & 14400.03 & 61.89\% & 99.94\% & 14400.03 & 49.50\% & 14400.01 \\
		& 3 & 59 & 743.18 & 100.00\% & 37.95 & 100.00\% & 14400.04 & 9.03\% & 100.00\% & 14400.02 & 33.21\% & 14400.02 \\
		& 5 & 98 & 810.79 & 100.00\% & 27.14 & 100.00\% & 14400.11 & 2.82\% & 100.00\% & 14400.05 & 85.76\% & 14400.02 \\
		& 7 & 137 & 834.51 & 100.00\% & 30.46 & 100.00\% & 14400.03 & 0.86\% & 100.00\% & 14400.03 & 99.27\% & 14400.01 \\
		& 10 & 195 & 844.27 & 100.00\% & 21.74 & 100.00\% & 9.57 & 0.00\% & 100.00\% & 812.21 & 100.00\% & 94.47 \\
		& 13 & 254 & 844.27 & 100.00\% & 14.77 & 100.00\% & 0.78 & 0.00\% & 99.99\% & 26.83 & 100.00\% & 19.76 \\
		& 15 & 293 & 844.27 & 100.00\% & 11.44 & 100.00\% & 0.77 & 0.00\% & 100.00\% & 25.56 & 100.00\% & 19.05 \\
		& 20 & 390 & 844.27 & 100.00\% & 11.69 & 100.00\% & 0.77 & 0.00\% & 100.00\% & 10.62 & 100.00\% & 12.25 \\
		& 25 & 488 & 844.27 & 100.00\% & 11.36 & 100.00\% & 0.77 & 0.00\% & 100.00\% & 9.69 & 100.00\% & 10.35 \\
		& 30 & 585 & 844.27 & 100.00\% & 11.08 & 100.00\% & 0.77 & 0.00\% & 100.00\% & 10.94 & 100.00\% & 10.45 \\
		\hline
	\end{tabular}
\end{sidewaystable}

\begin{table}
	\centering
	\footnotesize
	\caption{Network flow restriction results on large test cases. For each instance, we show results for 10 different budgets for the percentage of relays that can be attacked. The load shed attained from the solution given by Algorithm~\ref{algo:1} is in the ``NFLB" column and the computational time to get NFLB is shown in the last column.\label{table:large-nf}}
	\begin{tabular}{c | r r | r r }
			\hline 
			Instance & Budget (\%) & \# Relays & NFLB & NFLB Time (s) \\
			\hline
			\multirow{10}{1.9cm}{2848rte}& 1 & 28 & 282.87 & 223.68 \\
			& 3 & 85 & 470.57 & 52.62 \\
			& 5 & 142 & 510.74 & 50.72 \\
			& 7 & 199 & 530.51 & 42.94 \\
			& 10 & 285 & 538.36 & 28.64 \\
			& 13 & 370 & 538.35 & 22.20 \\
			& 15 & 427 & 538.39 & 14.30 \\
			& 20 & 570 & 538.39 & 13.08 \\
			& 25 & 712 & 538.39 & 12.50 \\
			& 30 & 854 & 538.39 & 17.30 \\
			\hline
			\multirow{10}{1.9cm}{3012wp}& 1 & 30 & 93.74 & 1302.35 \\
			& 3 & 90 & 205.68 & 51.01 \\
			& 5 & 151 & 254.73 & 35.87 \\
			& 7 & 211 & 267.99 & 27.33 \\
			& 10 & 301 & 271.73 & 19.61 \\
			& 13 & 392 & 271.96 & 13.51 \\
			& 15 & 452 & 271.96 & 12.11 \\
			& 20 & 602 & 271.96 & 11.87 \\
			& 25 & 753 & 271.96 & 15.85 \\
			& 30 & 904 & 271.96 & 15.63 \\
			\hline
			\multirow{10}{1.9cm}{3375wp}& 1 & 34 & 203.44 & 311.63 \\
			& 3 & 101 & 372.36 & 225.19 \\
			& 5 & 169 & 470.08 & 64.40 \\
			& 7 & 236 & 505.73 & 49.03 \\
			& 10 & 337 & 519.36 & 37.04 \\
			& 13 & 439 & 520.99 & 31.44 \\
			& 15 & 506 & 520.99 & 20.61 \\
			& 20 & 675 & 520.99 & 20.68 \\
			& 25 & 844 & 520.99 & 20.29 \\
			& 30 & 1012 & 520.99 & 19.68 \\
			\hline
			\multirow{10}{1.9cm}{6468rte}& 1 & 65 & 554.61 & 899.85 \\
			& 3 & 194 & 825.38 & 294.70 \\
			& 5 & 323 & 889.38 & 243.10 \\
			& 7 & 453 & 924.60 & 222.20 \\
			& 10 & 647 & 948.15 & 229.43 \\
			& 13 & 841 & 951.54 & 86.39 \\
			& 15 & 970 & 951.54 & 80.63 \\
			& 20 & 1294 & 951.57 & 48.33 \\
			& 25 & 1617 & 951.57 & 45.94 \\
			& 30 & 1940 & 951.57 & 44.39 \\
			\hline
	\end{tabular}
\end{table}

\subsection{Quality of NFLB}

Without a nontrivial upper bound on the worst-case relay attack problem, we cannot comment precisely on the quality of the network flow restriction.
However, in comparisons with the lower bound attained from solving with a heuristic bound on the dual variables, we see that in 113 out of 120 instances, NFLB was the best bound.
In the 7 instances where NFLB was not the best lower bound, it was $89.25\%$,  $91.73\%$, $99.77\%$, $99.06\%$, $99.88\%$, $99.31\%$, and $97.60\%$ of the best load shed found.
The budgets for which there is a gap between the best-known solution and the network flow restriction solution tend to be small.
This is consistent with the bound from Theorem \ref{thm:thermal-limit-bound} since for these budgets there is relatively little load shed, meaning that the $\ell_1$- norm of the injections is likely quite large relative to its maximum possible value for the instance (when all the load is served), making the right-hand side of (\ref{eq:the-bound}) large.
In this case, the theory suggests that network flow is not as good of an approximation of DCOPF.
However, at least for smaller network sizes, Gurobi is able to solve (\ref{relay-attack-single-level-linear}) for smaller attack budgets with a heuristic bound on the dual variables, and might be a better option.
For larger network sizes, even though we sometimes see a slight gap between the NFLB and (\ref{relay-attack-single-level-logical}) or (\ref{relay-attack-single-level-linear}), the network flow solution still appears to be of extremely good quality.
Additionally, for these networks, (\ref{relay-attack-single-level-logical}) and (\ref{relay-attack-single-level-linear}) do not scale well enough to be computationally tractable: Among the 90 larger instances tested, (\ref{relay-attack-single-level-linear}) fails to achieve the NFLB within 4 hours in 25 instances, and (\ref{relay-attack-single-level-logical}) fails to do so in 39 instances.
Last, note that even in the congested variations of the test networks shown in Tables \ref{table:3} and \ref{table:4}, the quality of NFLB is good despite the theoretical results not holding.

\subsection{Computational Tractability of Algorithm~\ref{algo:1}}

In Tables \ref{table:1}, \ref{table:2}, \ref{table:3}, and \ref{table:4}, Algorithm~\ref{algo:1} takes less than 4 minutes on all of the instances of the problem tested.
In Table \ref{table:large-nf}, Algorithm~\ref{algo:1} takes less than 22 minutes in all cases, and often takes less than 5 minutes.
In contrast, when solving (\ref{relay-attack-single-level-linear}), Gurobi times out without proving optimality within 4 hours for the hardest budgets on all but the smallest test case. 
For the nine larger cases, Gurobi takes more than 4 hours to find a solution of the same quality as NFLB using problem (\ref{relay-attack-single-level-logical}) and with the heuristic value of $M$ in problem (\ref{relay-attack-single-level-linear}).
In essence, we see that scaling up the size of the network for difficult attack budgets is not feasible solving a linearization of the single-level reformulation of (\ref{relay-attack-prob}), even with small heuristic bounds on the duals.
However, we can easily find what we believe to be a good-quality solution for even a 6,468 bus network using Algorithm~\ref{algo:1}.

\begin{figure}
	\centering
	\subfloat[Computational times for solving (\ref{relay-attack-single-level-linear}) with heuristic $M$.]{
		\includegraphics[width=0.5\linewidth]{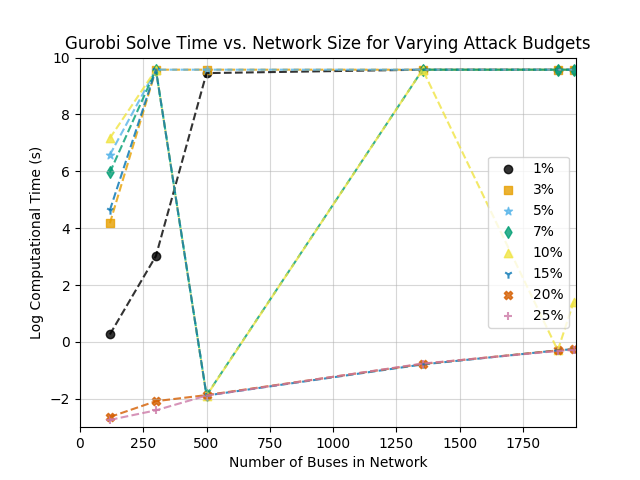}
		\label{sfig:gurobi-scaling}
	}
	\subfloat[Computational times for the network flow restriction.]{
		\includegraphics[width=0.5\linewidth]{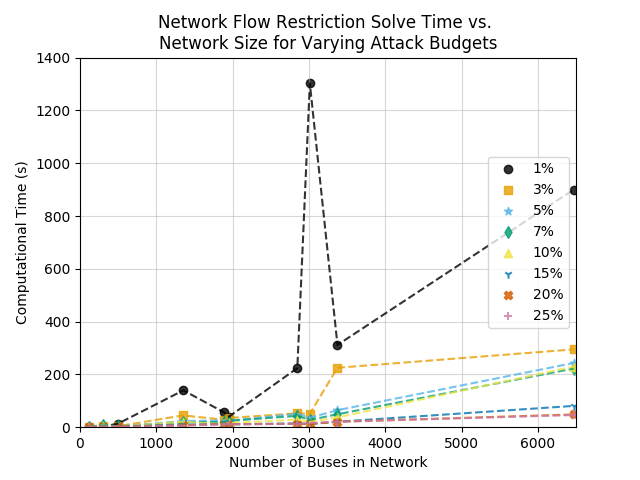}
		\label{sfig:nf-scaling}
	}
	\caption{Computational times for various budgets for both Gurobi with the dual bound set to 2 and for the network flow restriction. The first plot shows the data from the ``Question 1: Problem (\ref{relay-attack-single-level-linear}) Time" column of Tables \ref{table:1} and \ref{table:2}, and the second plot shows the data from the ``NFLB Time" column from Tables \ref{table:1}, \ref{table:2}, and \ref{table:large-nf}. \label{fig:scaling}}
\end{figure}

These observations are visualized in Figure \ref{fig:scaling}:
In Figure \ref{sfig:gurobi-scaling} we plot on a log scale the computational times to solve (\ref{relay-attack-single-level-linear}) linearized using the heuristic value of $M$.
There is noise in the 7\% and 10\% budgets because the most difficult budgets in that range depend on the particular network, not just the number of nodes.
However, in general we see that, even for the easier very large budgets, the solve times appear to scale exponentially.
For the smaller, more difficult budgets, we hit the 4-hour time limit for most of the networks.
In Figure \ref{sfig:nf-scaling}, we plot computational times for all ten of our test networks on a linear scale.
The scaling for this method appears to be roughly linear in the size of the network, where the lower budgets are more difficult and the higher budgets tend to be easier.
The spike for the 1\% budget on the 3,012 bus instance is consistent for different seeds: It appears to be an anomaly in terms of difficulty for Gurobi.

Overall, we find NFLB to be approximately 150 times faster than using Gurobi. 
	We arrive at this number by taking the average of the ratio of the time for Gurobi to reach NFLB using Problem (\ref{relay-attack-single-level-linear}) and the time to compute NFLB over the 120 instances tested.
In summary, we see through our computational experiments that the most difficult instances of the worst-case relay attack problem are for mid-range budgets on large networks.
Solving the traditional linearized single-level formulation including DCOPF in the inner problem does not scale well, even when the bound on the duals is as small as 2 or 3.
In contrast, we are able to solve challenging budgets on networks up to 6,468 nodes in less than 25 minutes using the network flow restriction.
To the extent it is ascertainable, the quality of solutions is good.

\section{Conclusion}\label{sec:conclusions}

In this work, we analyzed a restriction of the worst-case relay attack problem which has theoretical guarantees on uncongested networks and which we have also shown empirically to provide a high-quality lower bound, even on congested networks.
We have shown that, in addition to the apparent tightness of the lower bound, the network flow restriction can be solved efficiently and to scale with a commercial MIP solver.
We suspect this is due in part to the fact that the network flow restriction can be linearized with big-M values of 1, and also in part to the familiar, well-studied structure of network flow itself.

In future work, there is a need to consider upper bounds for this problem and to improve the scalability of exact methods. 
Additionally, higher-complexity restoration models have been shown to be important for $N-k$ models when $k$ is large (\cite{CoffrinBTSB2019}), so there is a need to find scalable solution methods when the defender problem includes elements such as nonlinear approximations of the AC power flow equations, bus shunts, and line charging. Last,
the network flow approximation for DCOPF could be used in place of DCOPF in numerous other problems for both power systems operations and security.
Since power systems are rarely congested in practice, it is likely that this approximation can be of use in order to scale up other problems which currently rely on DCOPF.

\section*{Acknowlegements}
We would like to thank Bryan Arguello, Anya Castillo, Jared Gearhart, and Cynthia Phillips for helpful discussions during this work.

\bibliography{relay_attack_neutral}

\begin{appendices}
\section{Single-Level Formulations of the Worst-Case Relay \\ Attack Problem}\label{sec:single-level-reformulations}
	In this appendix, we give the two single-level formulations of (\ref{relay-attack-prob}) that we compare against problem (\ref{network-flow-single-level-linear}).
	Let $\xi^+$ and $\xi^-$ represent the duals of constraints (\ref{eq:ohms-nonlinear-ub}) and (\ref{eq:ohms-nonlinear-lb}) respectively, and $\kappa^+$ and $\kappa^-$ be the duals of the phase angle bound constraints in (\ref{eq:angle-bounds}).
	Then we have the logical formulation:
	\begin{subequations}\label{relay-attack-single-level-logical}
		\begin{align}
		\max \quad &\begin{aligned}
		& - \sum_{k \in \mathcal K} \left [ F_k ( \overline \lambda_k^+ + \overline \lambda_k^-) + 2\pi B_k (\overline \xi_k^+ + \overline \xi^-_k ) \right ] \\
		&- \sum_{g \in \mathcal G} \overline P_g \overline \gamma_g + \sum_{b \in \mathcal B} \left [ D_b(\overline  \alpha_b + \mu_b - \beta_b) - \pi(\kappa_b^+ + \kappa_b^-) \right ]
		\end{aligned} \\
		\text{s.t.} \quad &\text{(\ref{eq:budget})-(\ref{eq:w-binary})} \nonumber \\
		&\lambda_k^+ - \lambda_k^- + \mu_{d(k)} - \mu_{o(k)} + \xi_k^+ - \xi_k^- = 0 & \forall k \in \mathcal K \quad & (f) \label{eq:flow-dual-with-ohms} \\
		& \mu_{b(g)} - \gamma_g \leq 0  &\forall g \in \mathcal G \quad & (p) \\
		& \alpha_b + \mu_b - \beta_b \leq 1  & \forall b \in \mathcal B \quad & (l) \\
		&\sum_{k \in \mathcal K_b^+} B_k (\xi_b^+ - \xi_b^-) + \sum_{k \in \mathcal K_b^-} (\xi_b^- - \xi_b^+) + \kappa_b^+ - \kappa_b^- = 0  & \forall b \in \mathcal B \quad &(\theta) \label{eq:theta-dual} \\
		& w_b = 0 \implies \overline \alpha_b = \alpha_b & \forall b \in \mathcal B \quad & \phantom{(f)} \label{eq:impl1} \\
		& w_b = 1 \implies \overline \alpha_b = 0 & \forall b \in \mathcal B \quad & \phantom{(f)} \\
		& v_k = 0 \implies \overline \xi_k^+ = \xi_k^+ & \forall k \in \mathcal K \quad & \phantom{(f)} \\
		& v_k = 1 \implies \overline \xi_k^+ = 0 & \forall k \in \mathcal K \quad & \phantom{(f)} \\
		& v_k = 0 \implies \overline \xi_k^- = \xi_k^- & \forall k \in \mathcal K \quad & \phantom{(f)} \\
		& v_k = 1 \implies \overline \xi_k^- = 0 & \forall k \in \mathcal K \quad & \phantom{(f)} \\
		& v_k = 0 \implies \overline \lambda_k^+ = 0 & \forall k \in \mathcal K \quad & \phantom{(f)} \\
		& v_k = 1 \implies \overline \lambda_k^+ = \lambda_k^+ & \forall k \in \mathcal K \quad & \phantom{(f)} \\
		& v_k = 0 \implies \overline \lambda_k^- = 0 & \forall k \in \mathcal K \quad & \phantom{(f)} \\
		& v_k = 1 \implies \overline \lambda_k^- = \lambda_k^- & \forall k \in \mathcal K \quad & \phantom{(f)} \\
		& u_g = 0 \implies \overline \gamma_g = 0 & \forall g \in \mathcal G \quad & \phantom{(f)} \\
		& u_g = 1 \implies \overline \gamma_g = \gamma_g & \forall g \in \mathcal G \quad & \phantom{(f)} \label{eq:impllast} \\
		& \xi^+_k \geq 0 \quad & \forall k \in \mathcal K \quad & \phantom{(f)} \label{eq:nonneg-dual-1}\\
		& \xi^-_k \geq 0 \quad & \forall k \in \mathcal K \quad & \phantom{(f)} \\
		& \alpha_b \geq 0 \quad &\forall b \in \mathcal B\quad & \phantom{(f)} \\
		& \beta_b \geq 0 \quad &\forall b \in \mathcal B\quad & \phantom{(f)}\\
		& \lambda_k^+ \geq 0 \quad &\forall k \in \mathcal K \quad & \phantom{(f)}\\
		& \lambda_k^- \geq 0 \quad &\forall k \in \mathcal K \quad & \phantom{(f)}\\
		& \gamma_g \geq 0 \quad &\forall g \in \mathcal G\quad & \phantom{(f)} \\
		& \kappa^+_b \geq 0 \quad & \forall b \in \mathcal B \quad & \phantom{(f)} \\
		& \kappa^-_b \geq 0 \quad & \forall b \in \mathcal B \quad & \phantom{(f)} \label{eq:nonneg-dual-last}
		\end{align}
	\end{subequations}
	We do not specify upper bounds on the dual variables of DCOPF and we encode constraints (\ref{eq:impl1})-(\ref{eq:impllast}) using Gurobi IndicatorConstraints.
	
	Next, we give a mixed integer linear programming reformulation of (\ref{relay-attack-single-level-logical}) with a heuristic upper bound on the dual variables.
	Let $M$ represent the heuristic bound chosen for the dual variables of the DCOPF problem.
	We use this bound to give a mixed integer linear representation of the implications in (\ref{eq:impl1})-(\ref{eq:impllast}):
	\begin{subequations}\label{relay-attack-single-level-linear}
		\begin{align}
		\max \quad &\begin{aligned}
		& - \sum_{k \in \mathcal K} \left [ F_k ( \overline \lambda_k^+ + \overline \lambda_k^-) + 2\pi B_k (\overline \xi_k^+ + \overline \xi^-_k ) \right ] \\
		&- \sum_{g \in \mathcal G} \overline P_g \overline \gamma_g + \sum_{b \in \mathcal B} \left [ D_b(\overline  \alpha_b + \mu_b - \beta_b) - \pi(\kappa_b^+ + \kappa_b^-) \right ]
		\end{aligned} \\
		\text{s.t.} \quad &\text{(\ref{eq:budget})-(\ref{eq:w-binary}), (\ref{eq:flow-dual-with-ohms})-(\ref{eq:theta-dual}), (\ref{eq:nonneg-dual-1})-(\ref{eq:nonneg-dual-last})} \nonumber \\
		& 0 \leq \overline \alpha_b \leq \alpha_b & \forall b \in \mathcal B \quad & \phantom{(f)} \\
		& \alpha_b - Mw_b \leq \overline \alpha_b \leq M(1 - w_b) & \forall b \in \mathcal B \quad & \phantom{(f)} \\
		& 0 \leq \overline \xi_k^+ \leq \xi_k^+ & \forall k \in \mathcal K \quad & \phantom{(f)} \\
		& \xi_k^+ - Mv_k \leq \overline \xi_k^+ \leq M(1 - v_k) & \forall k \in \mathcal K \quad & \phantom{(f)} \\
		& 0 \leq \overline \xi_k^- \leq \xi_k^- & \forall k \in \mathcal K \quad & \phantom{(f)} \\
		& \xi_k^- - Mv_k \leq \overline \xi_k^- \leq M(1 - v_k) & \forall k \in \mathcal K \quad & \phantom{(f)} \\
		& 0 \leq \overline \lambda_k^+ \leq \lambda_k^+ & \forall k \in \mathcal K \quad & \phantom{(f)} \\
		& \lambda_k^+ - M(1 - v_k) \leq \overline \lambda_k^+ \leq Mv_k & \forall k \in \mathcal K \quad & \phantom{(f)} \\
		& 0 \leq \overline \lambda_k^- \leq \lambda_k^- & \forall k \in \mathcal K \quad & \phantom{(f)} \\
		& \lambda_k^- - M(1 - v_k) \leq \overline \lambda_k^- \leq Mv_k & \forall k \in \mathcal K \quad & \phantom{(f)} \\
		& 0 \leq \overline \gamma_g \leq \gamma_g & \forall g \in \mathcal G \quad & \phantom{(f)} \\
		& \gamma_g - M(1 - u_g) \leq \overline \gamma_g \leq Mu_g & \forall g \in \mathcal G \quad & \phantom{(f)}
		\end{align}
	\end{subequations}

\section{Proofs of Theorems}\label{sec:proofs}

\subsection{Proof of Observation \ref{obs:dual-bounds-are-1}.}\label{sec:proof-dual-bounds-are-1}
Let $I_k$ be the $k\times k$ identity matrix. 
Without loss of generality we may assume that there is exactly one generator per bus.
We can do this because we already assumed the generator dispatch lower bound is 0, so if there are multiple generators at a bus, we can aggregate them into one by summing their maximum capacities.
Note that this means $|\mathcal G| = |\mathcal B|$ in the following.
To show the claim, we will show that the elements of any extreme point of the dual polyhedron, defined by (\ref{nf:dual-flow})-(\ref{nf:dual-actually-last}), are bounded in absolute value by 1.
For notational convenience, let $\{x : Gx \leq h\}$ be the system (\ref{nf:dual-flow})-(\ref{nf:dual-actually-last}).
Note that $G$ is an integer matrix.
Let $n = 3|\mathcal B| + 2|\mathcal K| + |\mathcal G|$, the dimension of the dual space.
Then an extreme point of the polyhedron is the feasible solution where a subsystem of $n$ inequalities from $Gx \leq h$ hold at equality.
Let $\overline G$ denote the square submatrix of $G$ corresponding to this subsystem.
By Cramer's Rule, this means that we can calculate the $i$th component of that solution:
\[
\hat x_i = \frac{\det \left ( \begin{bmatrix}\overline G^1 & \overline G^2 & \cdots & \overline G^{i-1} & h & \overline G^{i+1} &\cdots& \overline G^n \end{bmatrix} \right )}{\det(\overline G)},
\]
where $\overline G^j$ is the $j$th column of $\overline G$.
Since $G$ is integer, we know that $\lvert \det(\overline G) \rvert \geq 1$.
This means that 
\begin{equation}\label{det-bound}
\lvert \hat x_i \rvert \leq \left \lvert \det \left ( \begin{bmatrix} \overline G^1 & \overline G^2 & \cdots & \overline G^{i-1} & h & \overline G^{i+1} &\cdots& \overline G^n \end{bmatrix} \right ) \right \rvert .
\end{equation}
In the following, we show that the right-hand side of (\ref{det-bound}) is 1 by showing the matrix in question is totally unimodular.

We will show that $\begin{bmatrix} G & h \end{bmatrix}$ is totally unimodular since that means any submatrix of $\begin{bmatrix} G & h\end{bmatrix}$ is totally unimodular.
Writing the columns corresponding to the ordering of the variables $(\lambda^+, \lambda^-, \mu, \gamma, \alpha, \beta)$, we can write
\begin{equation}\label{eq:Gh}
\begin{bmatrix} G & h\end{bmatrix} = 
\begin{bmatrix}
I_{|\mathcal K |} & -I_{|\mathcal K |} & N^\top & 0 & 0  & 0 & 0\\
- I_{|\mathcal K |} & I_{|\mathcal K |} & - N^\top & 0 & 0  & 0 & 0\\
0 & 0 & I_{|\mathcal G|} & - I_{|\mathcal G|}  & 0 & 0 & 0 \\
0 & 0 & I_{|\mathcal B|} & 0 & I_{|\mathcal B|} & - I_{|\mathcal B|} & \boldsymbol{1} \\
0 & 0 & 0 & 0 & -I_{|\mathcal B|} & 0 & 0\\
0 & 0 & 0 &  0 & 0 &  -I_{|\mathcal B|} & 0\\
-I_{|\mathcal K|} & 0 & 0 & 0 & 0 & 0 & 0 \\
0 & -I_{|\mathcal K|} & 0 & 0 & 0 & 0 & 0\\
0 & 0 & - I_{|\mathcal G|} & 0 & 0 & 0 & 0\\
\end{bmatrix}
\end{equation}
where, without loss of generality, we relabel the generators so that we get the identity in the part of the matrix corresponding to $\mu$ in constraints (\ref{nf:dual-dispatch}).
We use $N$ to represent the node-arc adjacency matrix of the network, which is known to be totally unimodular.
We use $\boldsymbol{1}$ to represent the vector of all 1's in $\mathbb R^{|\mathcal B|}$.
From (\ref{eq:Gh}), we see it suffices to show that 
\[
A = \begin{bmatrix}
N^\top & 0 \\
I_{|\mathcal G|} & 0 \\
I_{|\mathcal B|} & \boldsymbol{1}
\end{bmatrix}
\]
is totally unimodular since $\begin{bmatrix} G & h \end{bmatrix}$ augments $A$ by a series of identities and the negative of the first row.
It is easy to verify that $A$ is totally unimodular, for example it satisfies the conditions of the theorem by Hoffman (\cite{HellerT1956}).

By the definition of total unimodularity, it follows from the claim above and equation (\ref{det-bound}) that $-1 \leq \hat x_i \leq 1$ for all $i$.
This means that 1 is a valid upper bound for $\alpha$, $\beta$, $\lambda^+$, $\lambda^-$, $\gamma$, and $\mu$, and -1 is a valid lower bound for $\mu$ in problem (\ref{network-flow-single-level}).

\subsection{Proof of Theorem \ref{thm:thermal-limit-bound}}\label{sec:proof-of-thm}

We will first establish some lemmas before we give a proof of Theorem \ref{thm:thermal-limit-bound}.
Let $N$ be the $|V|\times|E|$ node-arc incidence matrix of a connected digraph $G(V, E)$.
We remind the reader of two facts:
\begin{itemize}
	\item The rank of $N$ is $|V| - 1$. 
	\item Let $N^{(i)}$ be the matrix where the $i$th row if $N$ is removed and let $N_i$ be the $i$th row of $N$.
	Then 
	\begin{equation}\label{eq:get-last-row}
	N_i = - \sum_j N_j^{(i)}.
	\end{equation}
	That is, we can calculate any given row of the matrix by taking the negative of the sum of the other rows.
\end{itemize}
In the following Lemma, we derive the injection shift factor formulation for DCOPF.
\begin{lemma}\label{lemma:isf}
	Consider a DCOPF problem over a connected digraph $G(V,E)$.
	Let $B \in \R^{|E|\times |E|}$ be a diagonal matrix where the diagonal entries are the susceptances of the lines.
	Let $0 \in V$ and consider the matrix $N^{(0)}$.
	Let $d \in \R^{|V|}$ be a set of injections which satisfy global balance, i.e.,
	\[
	\sum_j d_j = 0.
	\]
	Let $d^{(0)}$ be the vector of injections where we have removed the component corresponding to the $0$th node.
	Then the unique vector of flows that satisfies (i) nodal balance constraints given the injections $d$ and (ii) Ohm's law is
	\[
	f = B(N^{(0)})^\top \left ( N^{(0)}B(N^{(0)})^\top \right )^{-1} d^{(0)}.
	\]
\end{lemma}
\begin{proof}
Let $\theta^{(0)}$ be the vector of phase angles where we have removed the component corresponding to the $0$th node.
Then $f$ and $\theta^{(0)}$ must satisfy
\begin{align}
N^{(0)}f &= d^{(0)} & \text{(nodal balance)} \label{eq:nodal-balance-matrix}\\
f &= B(- N^{(0)})^\top \theta^{(0)} & \text{(Ohm's law)} \label{eq:ohms-law-matrix}
\end{align}
Combining (\ref{eq:nodal-balance-matrix}) and (\ref{eq:ohms-law-matrix}), 
\[
-N^{(0)}B(N^{(0)})^\top \theta^{(0)} = d^{(0)}.
\]
Since $N^{(0)}$ is full row rank and $B$ is a diagonal matrix with all entries positive, $N^{(0)}B(N^{(0)})^\top$ is invertible.
This means
\begin{align}
\theta^{(0)} &= -\left (N^{(0)}B(N^{(0)})^\top \right )^{-1} d^{(0)} \\
\Rightarrow f &= B (N^{(0)})^\top \left (N^{(0)}B(N^{(0)})^\top \right )^{-1} d^{(0)}.
\end{align}
Set the phase angle of the $0$th node to 0.
Then by (\ref{eq:get-last-row}), the resulting phase angles and the vector of flows above satisfy nodal balance constraints and Ohm's law constraints.
\end{proof}

Given a square matrix $H$, let $\lambda_{\max}(H)$ be the largest eigenvalue of $H$.
\begin{lemma}\label{lemma:orthogonal-projection}
	Let $A \in \R^{m \times n}$ be a matrix with full row rank.
	Then $\lambda_{\max}(A^\top (A A^\top)^{-1}A) = 1$.
\end{lemma}
\begin{proof}
The matrix $A^\top(AA^\top)^{-1}A$ is an orthogonal projection matrix, so all of its eigenvalues are 1 or 0 (since it is idempotent).
Since $A$ has rank $m$, so does $A^\top(AA^\top)^{-1}A$, so $m$ of them are 1, and we have the result.
\end{proof}

\begin{lemma}\label{lemma:bound-l1}
	Let $d \in \R^n$ be such that $\sum_{i = 1}^n d_i = 0$. Then $|\sum_{i \in S} d_i | \leq \frac 1 2 \Vert d\Vert _1$ for all $S \subseteq \{1, 2, \dots n\}$
\end{lemma}
\begin{proof}
The statement is trivially true for $S = \emptyset$ and $S = \{1, 2, \dots, n\}$.
So assume $\emptyset \subsetneq S \subsetneq \{1, 2, \dots, n\}$.
Note first that $\sum_{i \in S} |d_i| + \sum_{i \in \{1, 2, \dots, n\}\setminus S} |d_i| = \Vert d\Vert _1$.
Then we have that
\begin{equation}\label{eq:min-less-than-half}
\min \bigg \{ \sum_{i \in S} |d_i|, \sum_{i \in \{1, 2, \dots, n\}\setminus S} |d_i| \bigg \} \leq \frac 1 2 \Vert d\Vert _1.
\end{equation}
By the triangle inequality we have,
\begin{align}\label{eq:less-sum-S}
\bigg \lvert \sum_{i \in S} d_i \bigg \rvert \leq \sum_{i \in S} |d_i|,
\end{align}
and since $\sum_{i = 1}^n d_i = 0$,
\begin{align}\label{eq:less-sum-not-S}
\bigg \lvert \sum_{i \in S} d_i \bigg \rvert = \bigg \lvert \sum_{i \in \{1, 2, \dots, n\}\setminus S} d_i \bigg \rvert \leq \sum_{i \in \{1, 2, \dots, n\}\setminus S} |d_i|.
\end{align}
Combining (\ref{eq:less-sum-S}) and (\ref{eq:less-sum-not-S}), we get
\[
\bigg \lvert \sum_{i \in S} d_i \bigg \rvert \leq \min \bigg \{ \sum_{i \in S} |d_i|, \sum_{i \in \{1, 2, \dots, n\}\setminus S} |d_i| \bigg \},
\]
so the result follows from (\ref{eq:min-less-than-half}).
\end{proof}

We are ready to prove the theorem.
\begin{proof}{Proof of Theorem \ref{thm:thermal-limit-bound}.}
Since $d$ is flow-polytope feasible, let $f^\nf \in \R^{|\mathcal K|}$ be the flow vector that satisfies thermal limits and nodal balance constraints given the node injection values $d$.
We must show that there exists a flow vector that not only satisfies nodal balance constraints given the injections $d$ and thermal limits, but also Ohm's law.

{\it{\bf Claim 1: } It is sufficient to prove the DCOPF polytope is non-empty on each of the subgraphs corresponding to $V^i$, where we may assume that the $\ell_1$-norm of the injections on the vertices $V^i$  is at most $\Vert d\Vert _1$.}

Claim 1 is straightforward to verify, so we only sketch the arguments here. 
For the arcs connecting vertex blocks $V^i$ and $V^j$ where $i \ne j$, we will keep the flow values from $f^\nf$.
That flow clearly satisfies the thermal limit, and since those arcs are not involved in any cycles, once we find flow values on the incident arcs within each $V^i$, we will be able to find values of $\theta$ such that Ohm's law will also be satisfied.
Thus, the problem reduces to finding flows within blocks of nodes $V^i$ for $i \in \{1, 2, \dots, m\}$.
It is straightforward then to show that the $\ell_1$-norm of the injections on the vertices $V^i$  is at most $\Vert d\Vert _1$.

Consider a block $V$ (we drop the superscript for simplicity), recalling that it has at most $r(G)$ nodes.
Let the net injections on the nodes be $d^V$ such that $\Vert d^V\Vert _1 \leq \Vert d\Vert _1$.
For simplicity of notation, we will refer to the subgraph on $V$ as $H(V,E)$.
Let $N \in \{0, 1, -1\}^{|V| \times |E|}$ be the node-arc incidence matrix of $H$.
Let $B \in \R^{|E|}$ be a diagonal matrix with $B_{ee}$ equal to the susceptance on arc $e$.
Let $v_0 \in V$ be an arbitrarily chosen reference bus, let $N^{(v_0)}$ be as defined before, and let $d^{(v_0)}$ be the vector where we have removed the component corresponding to $v_0$ from $d^V$.
Then by Lemma \ref{lemma:isf}, the unique flow that satisfies the DCOPF constraints on block $V$ is
\begin{equation}\label{eq:calc-flows}
f = B(N^{(v_0)})^\top \left ( N^{(v_0)}B(N^{(v_0)})^\top \right )^{-1} d^{(v_0)}.
\end{equation}
Let $\sqrt B$ be a diagonal matrix whose $(e, e)$th entry is $\sqrt{B_{ee}}$.

{\it {\bf Claim 2: } There exists a vector $\tilde d \in \R^{|E|}$ such that
	\[
	(N^{(v_0)}\sqrt B) \tilde d = d^{(v_0)}, \text{ and } \Vert \tilde d\Vert _2 \leq \frac{\sqrt{r(G)-1}}{2\sqrt{B_{\min}}}\Vert d^V\Vert _1.
	\]}
We will show that there exists $x \in \R^{|E|}$ such that $N^{(v_0)}x = d^{(v_0)}$ and $\Vert x\Vert _2 \leq \frac{\sqrt{r(G)-1}}{2}\Vert d^V\Vert _1$.
This completes the proof since we can then find $\tilde d$ by solving $\sqrt B \tilde d = x$.
In the solution, we will have $\tilde d \leq \frac 1 {\sqrt{B_{\min}}} x$, since $\sqrt B$ is a diagonal matrix and $\sqrt{B_{\min}}$ is the smallest diagonal entry.
This means that 
\[
\Vert \tilde d\Vert _2 \leq \frac 1 {\sqrt{B_{\min}}}\Vert x\Vert _2 \leq \frac{\sqrt{r(G)-1}}{2\sqrt{B_{\min}}} \Vert d^V\Vert _1,
\]
and $(N^{(v_0)} \sqrt B)\tilde d = N^{(v_0)}x = d^{(v_0)},$ as required.

Let $N^{(v_0)} = \begin{bmatrix} P & Q \end{bmatrix}$ where $P$ is composed of columns corresponding to the arcs of a spanning tree in $H(V,E)$.
This means that $P$ is a full row rank square matrix, and furthermore that it is totally unimodular since it is the adjacency matrix of a bipartite graph.
Solve
\begin{equation}\label{eq:tree-flow}
P \hat d = d^{(v_0)}
\end{equation}
and let $x = \begin{bmatrix} \hat d \\ 0 \end{bmatrix}$.
So it is sufficient to show that $\Vert \hat d \Vert _2 \leq \frac{\sqrt{r(G)-1}}{2} \Vert d^V\Vert _1$.
Note that, by (\ref{eq:tree-flow}), $\hat d$ is a flow on the tree corresponding to $P$ where the injections on the nodes are given by $d^{(v_0)}$.
Note that since $P$ represents a tree, the removal of any arc of the graph disconnects the graph.
If we remove arc $i$, let $S_i$ represent the set of nodes in the component containing $o(i)$.
Using this notation, this means that for all arcs $i$, $\hat d_i = \sum_{j \in S_i} d_j^V$.
By Lemma \ref{lemma:bound-l1}, this means that $|\hat d_i| \leq \frac 1 2 \Vert d^V\Vert _1$.
Finally, the support of $\hat{d}$ is at most $r(G) -1$, since a tree on $r(G)$ nodes has $r(G)-1$ arcs.
So $\Vert \hat d\Vert _2 = \sqrt{\sum_{i = 1}^{r(G) - 1} \hat d_i^2} \leq \frac{\sqrt{r(G)-1}}{2}\Vert d^V\Vert _1$, showing Claim 2.

Now, we can rewrite (\ref{eq:calc-flows}) as 
\begin{align*}
f &= B(N^{(v_0)})^\top \left ( N^{(v_0)}B(N^{(v_0)})^\top \right )^{-1} d^{(v_0)} \\
&= B(N^{(v_0)})^\top \left ( N^{(v_0)}B(N^{(v_0)})^\top \right )^{-1} N^{(v_0)}\sqrt B \tilde d \\
&= \sqrt B \sqrt B(N^{(v_0)})^\top \left ( N^{(v_0)}\sqrt B \sqrt B(N^{(v_0)})^\top \right )^{-1} N^{(v_0)}\sqrt B \tilde d \\
&= \sqrt B (N^{(v_0)}\sqrt B)^\top \left ( (N^{(v_0)}\sqrt B) (N^{(v_0)}\sqrt B)^\top \right )^{-1} (N^{(v_0)}\sqrt B) \tilde d
\end{align*}
where the second equality holds by Claim 2 and the last holds since $\sqrt B$ is symmetric.
Therefore
\begin{align*}
\Vert f\Vert _2 &\leq \lambda_{\max} \left ( \sqrt B (N^{(v_0)}\sqrt B)^\top \left ( (N^{(v_0)}\sqrt B) (N^{(v_0)}\sqrt B)^\top \right )^{-1} (N^{(v_0)}\sqrt B)\right) \Vert \tilde d\Vert _2 \\
&\leq \lambda_{\max} ( \sqrt B )\lambda_{\max}\left (N^{(v_0)}\sqrt B)^\top \left ( (N^{(v_0)}\sqrt B) (N^{(v_0)}\sqrt B)^\top \right )^{-1} (N^{(v_0)}\sqrt B) \right ) \Vert  \tilde d \Vert _2 \\
&\leq \sqrt{B_{\max}} \cdot 1 \cdot \frac{\sqrt{r(G)-1}}{2\sqrt{B_{\min}}}\Vert d^V\Vert _1,
\end{align*}
where the last inequality follows Lemma \ref{lemma:orthogonal-projection} and from Claim 2.

By the assumption of the theorem, $\overline F_e \geq \sqrt{\frac{B_{\max}}{B_{\min}}} \frac{\sqrt{r(G)-1} }{2}\Vert d\Vert _1$, for all $e \in E$.
Thus the $f$ above is feasible, completing the proof.
\end{proof}

\subsection{Proof of Corollary \ref{corollary:approximation-works}}\label{sec:proof-approximation-works}

By construction, problems (\ref{eq:dcopf}) and (\ref{eq:nf-relaxation}) are both bounded and feasible (since it is always possible to shed all the load and since $l$ is bounded).
Also, since (\ref{eq:nf-relaxation}) is a relaxation of (\ref{eq:dcopf}), $z^l \leq z^*$.
It is sufficient to show that there exists a solution to (\ref{eq:dcopf}) with the same objective value as (\ref{eq:nf-relaxation}).

Let $(\tilde f, \tilde l, \tilde p)$ be an optimal solution to (\ref{eq:nf-relaxation}).
Since $0 \leq \tilde l \leq D$ and $0 \leq \tilde p \leq \overline P$, it is sufficient to show that the system
\begin{equation}\label{eq:flows}
\begin{aligned}
\sum_{k \in \mathcal K^+(b)} f_k - \sum_{k \in \mathcal K^-(b)} f_k &= D_b - \tilde l_b - \sum_{g \in \mathcal G_b} \tilde p_g & \forall b \in \mathcal B\\
f_k &= B_k(\theta_{o(k)} - \theta_{d(k)}) & \forall k \in \mathcal K \\
f_k &\geq -\overline F_k & \forall k \in \mathcal K \\
-f_k &\geq -\overline F_k & \forall k \in \mathcal K 
\end{aligned}
\end{equation}
has a feasible solution.
Since $\tilde p \geq 0$ and since the injections satisfy global balance, that is
\begin{equation}\label{eq:global-balance}
\sum_{b \in \mathcal B} (D_b - \tilde l_b) = \sum_{g \in \mathcal G} \tilde p_g,
\end{equation}
we have that
\[
\Vert D - \tilde l - \tilde p\Vert_1 \leq \Vert D - \tilde l\Vert_1 + \Vert\tilde p\Vert_1  =  2\Vert D - \tilde l\Vert_1 \leq 2 \Vert D\Vert_1,
\]
where the first inequality follows from the triangle inequality, the equality comes from  (\ref{eq:global-balance}) and the fact that $\tilde l \leq D$ and $0 \leq \tilde p$, and the last inequality again follows from $\tilde l \leq D$. 
Since 
\[
\overline F_k \geq 2 \sqrt{\frac{B_{\max}}{B_{\min}}} \frac{\sqrt{r(G) - 1}}{2} \Vert D\Vert _1 \geq \sqrt{\frac{B_{\max}}{B_{\min}}} \frac{\sqrt{r(G) - 1}}{2} \Vert D - \tilde l - \tilde p\Vert _1,
\]
the feasibility of system (\ref{eq:flows}) follows from Theorem \ref{thm:thermal-limit-bound}, completing the proof.
\qed

\subsection{Proof of Proposition \ref{prop:tight-within-constant}.}\label{sec:proof-tight-within-constant}
Let $n \in \N$ be given.
We construct $G^n(V^n, E^n)$ as follows:
\begin{itemize}
	\item $|V^n| = 3n$. Number the nodes from 1 to $3n$.
	\item For $i = 1, 2, \dots, 3n$, define the injections as follows:
	\[
	d^n_i = \begin{cases}
	-n & \text{ if } i = 3n \\
	1 & \text{ if } 1 \equiv i \pmod 3 \\
	0 & \text{ otherwise.}
	\end{cases}
	\]
	\item Let 
	\[
	E^n = \left ( \bigcup_{i = 0}^{n-1} \{(3i+1, 3i+2), (3i+2, 3i+3), (3i+1, 3i+3)\} \right ) \cup \left ( \bigcup_{i = 0}^{n-2} \{(3i+3, 3i+4)\} \right )
	\]
\end{itemize}
That is, $G^n$ is composed of $n$ triangles and no other cycles.
As an example, the digraph for $n = 3$ is shown in Figure \ref{artifical-graph-n3}.
By construction, for all such digraphs $G^n(V^n, E^n)$, $r(G^n) = 3$.
Let $c = \frac 1 {2\sqrt 3}$.
Let 
\[
\overline F_e = c \sqrt{\frac{B_{\max}}{B_{\min}}}\frac{\sqrt{r(G^n) - 1}}{2} \Vert d^n\Vert _1 = \frac 1 {2 \sqrt 3}\cdot \frac{\sqrt 3}{2}\cdot \left ( 2n \right ) = \frac n 2
\]
for all that non-cut-arcs, i.e., $e \in \{(3i+1, 3i+2), (3i +2, 3i+3), (3i+1, 3i+3) | i = 0, 1, \dots, n-1\}$.
Let
$
\overline F_e = n
$
for the remaining arcs, i.e., $e \in \{(3i+3, 3i+4)|i = 0, 1, \dots, n-2\}$.

There is a flow-polytope feasible flow on $G^n$ given by
\[
f_{(3i+1, 3i+2)} = f_{(3i+2, 3i+3)} = f_{(3i+1, 3i+3)} = \frac{i+1}{2}
\]
for $i = 0, 1, \dots, n-1$ and
\[
f_{(3i+3, 3i+4)} = i+1
\]
for $i = 0, 1, \dots, n-2$.
However, we can show that these injections are not DCOPF-feasible.
To see this, consider the triangle formed by nodes $3n-2$, $3n-1$, and $3n$.
We know that we have an in-flow of $n$ units to node $3n-2$, and that all of it must be routed to node $3n$.
Without loss of generality, suppose the phase angle at node $3n$ is 0.
There are exactly two paths from $3n-1$ to $3n$: the arc between them, and the two-arc path via $3n-1$.
Each of these paths has capacity $n/2$, so we must use both paths at capacity.
However, this is impossible, as it requires setting the phase angle at node $3n-1$ to $n/2$ and setting $\theta_{3n-2}$ to $n$. 
But that means the flow on the arc $(3n-2, 3n)$ is $n > \overline F_{(3n-2, 3n)} = \frac n 2$.
\qed

\end{appendices}

\end{document}